\documentclass[12pt]{amsart}
\usepackage{amssymb,amsfonts,amsmath}
\usepackage[all]{xy}  

\reversemarginpar 
\normalmarginpar 

\let\oldmarginpar\marginpar 
\renewcommand\marginpar[1]{\-\oldmarginpar{\raggedright\small\sf #1}}

\newcommand{\nc}{\newcommand}
\nc{\rc}{\renewcommand}

\nc{\C}{\mathbb{C}}
\nc{\Z}{\mathbb{Z}}
\nc{\N}{\mathbb{N}}
\nc{\Q}{\mathbb{Q}}
\nc{\R}{\mathbb{R}}
\rc{\t}{\mathrm}
\nc{\RR}{\mathcal{R}}
\nc{\PP}{\mathcal{P}}
\nc{\DD}{\mathcal{D}'}
\rc{\O}{\mathrm O}
\rc{\P}{\mathbb{P}}
\nc{\A}{\mathbb{A}}
\nc{\T}{\mathbb{T}}
\nc{\normx}{\parallel x \parallel}

\nc{\s}{\mathcal}
\rc{\t}{\mathrm}
\nc{\ts}{\t{L}_{\t{t}}}
\nc{\hs}{\t{L}_{\t{h}}}
\nc{\normm}{\parallel m \parallel}
\nc{\I}{\mathbf{1}}
\nc{\abss}{\mathrm{L}^1}  
\nc{\CC}{\mathcal{C}} 
\nc{\Real}{\mathrm{Re}}
\nc{\fun}{\t{Fun}}
\nc{\FF}{\mathcal{F}}
\nc{\GG}{\mathcal{G}}  
\nc{\HH}{\mathcal{H}}
\nc{\WW}{\mathcal{W}}
\nc{\UU}{\mathcal{U}}
\nc{\MM}{\mathcal{M}}
\nc{\VV}{\mathcal{V}}
\rc{\SS}{\mathcal{S}}
\rc{\sum}[1]{\underset{#1}{\Sigma}}
\nc{\pd}{\partial}
\nc{\D}{\Delta}
\nc{\n}{\parallel}
\nc{\normy}{\n y\n}
\nc{\G}{\Gamma}
\nc{\res}{\mathrm{R}}
\nc{\x}{\exp 2\pi \sqrt{-1}}
\nc{\e}{\epsilon}
\nc{\E}{\mathbf{e}}
\nc{\g}{\gamma}
\nc{\FD}{\mathrm{FD}}
\nc{\df}{$\mathrm{C}^{\infty}$}
\nc{\diff}{\mathrm{C}}
\nc{\rno}{\R^n\setminus\{0\}}
\nc{\units}{^{\times}}
\nc{\inverse}{^{-1}}
\nc{\reg}{\mathrm{reg}}
\rc{\1}{\mathbb{1}}
\nc{\sph}{\mathrm{S}}
\nc{\zn}{\Z^n}
\nc{\cstarn}{(\C\units)^n}
\nc{\sstarn}{(\sph^1)^n}
\nc{\qstar}{\Q\units}
\rc{\L}{\mathrm{L}^2}
\nc{\lrtn}{\L(\R^{2n}//\Z^{2n})}
\nc{\M}{\mathrm{M}}
\nc{\BB}{\mathcal{B}}
\nc{\hkr}{\hookrightarrow}
\nc{\EE}{\mathcal{E}}
\nc{\LL}{\mathcal{L}}
\rc{\SS}{\mathcal{S}} 
\nc{\Sing}{\mathrm{Sing}}
\nc{\KK}{\mathcal{K}} 
\nc{\hinf}{\HH_{\infty}}
\nc{\II}{\mathcal{I}}
\nc{\XX}{\mathcal{X}}
\nc{\YY}{\mathcal{Y}}
\nc{\OO}{\mathcal{O}} 
\nc{\WR}{W_{\R}}
\nc{\double}{(\rho_k (\Q^{\times}),\Q^m)}
\nc{\smooth}{\mathrm{C}^{\infty}}
\nc{\sms}{\mathrm{C}^{\infty}(m,s)}
\nc{\smsn}{\mathrm{C}^{\infty}(m,s)_n}
\nc{\cin}{\mathrm{C}^{\infty}(\R^{2n}//\Z^{2n})}
\nc{\din}{\DD(\R^{2n}//\Z^{2n})}
\nc{\wbar}{\overline{w}}
\nc{\tmp}{\mathcal{\SS}'(\R^n)}
\nc{\sch}{\SS(\R^n)}
\nc{\tpi}{2\pi\sqrt{-1}}
\rc{\a}{\tilde{a}}
\rc{\setminus}{\smallsetminus}
\nc{\wh}{\widehat}
\nc{\mc}{\mathcal}
\nc{\mf}{\mathbf}
\nc{\mb}{\mathbb}
\nc{\lr}{\longrightarrow}
\nc{\Hom}{\mathrm{Hom}}
\nc{\Ext}{\mathrm{Ext}}
\rc{\ll}{\longleftarrow}
\nc{\la}{\leftarrow}

\nc{\xr}{\xrightarrow}
\nc{\xl}{\xleftarrow}

\nc{\hr}{\hookrightarrow}

\nc{\dirlim}{\varinjlim}   
\nc{\invlim}{\varprojlim}  
\nc{\sx}{\s{S}(X)}
\nc{\sxp}{\s{S}(X')}
\nc{\ltx}{\t{L}^2(X)}
\nc{\ltxp}{\t{L}^2(X')}
\nc{\fb}{\s{F}_B}
\nc{\fbp}{\s{F}_{B'}}
\nc{\ltdouble}{\t{L}^2(X\times X'//\G\times\G')}
\nc{\smdouble}{\t{C}^{\infty}(X\times X'//\G\times \G')}
\nc{\disdouble}{\s{D}(X\times X'//\G\times \G')}
\nc{\ltdoublep}{\t{L}^2(X'\times X//\G'\times\G)}
\nc{\smdoublep}{\t{C}^{\infty}(X'\times X//\G'\times \G)}
\nc{\disdoublep}{\s{D}(X'\times X//\G'\times \G)}
\rc{\o}{\overline}
\nc{\test}{\t{C}^{\infty}_c}
\rc{\dot}[2]{\langle#1,#2\rangle}
\nc{\sgg}{\underset{\g\in\G}{\Sigma}}
\nc{\supp}{\t{supp}}
\nc{\semi}[3]{\|#1\|_{(#2,#3)}}
\nc{\semitest}[3]{\|#1\|^{(#2,#3)}}

\newtheorem{thm}{Theorem}[section]  
\newtheorem{lem}[thm]{Lemma}
\newtheorem{prop}[thm]{Proposition}
\newtheorem{cor}[thm]{Corollary}

\theoremstyle{remark}

\theoremstyle{definition}  
\newtheorem{rem}[thm]{Remark}

\newtheorem{dfn}[thm]{Definition} 
\newtheorem{ex}[thm]{Example}
 
\newtheorem{notn}[thm]{Notation}

\newtheorem{For}[thm]{Formalism and Notation}

\title{Summation and the Poisson formula}
\author{Madhav V. Nori}
\address{The University of Chicago \ 5734 S. University Ave \ Il 60637, USA}
\email{nori@math.uchicago.edu}
\dedicatory{Dedicated to T.N.Shorey}
\begin{document}
\maketitle
\begin{abstract}By giving the definition of the sum of a series indexed
by a set on which a group operates, we prove that the sum of the series that
defines the Riemann zeta function, the Epstein zeta function, and a few other
series indexed by $\Z^k$ has an intrinsic meaning as a complex number,
independent of the requirements of analytic continuation.
 The definition of the sum requires nothing
more than algebra and the concept of absolute convergence. The analytical
significance of the algebraically defined sum is then explained by an
argument that relies on the Poisson formula for tempered distributions. 
\end{abstract}

\section{Introduction}
We introduce two notions of summability of infinite series, t-summability (`t' stands for `translation') and h-summability ( `h' stands for `homothety'), with h-summability being more general. The definition of t-summability is given in \S1.1.  In \S1.2, we motivate h-summability with an example. 
The complete definition of h-summability is given in \S2.3.

The very first sentence of \S1.2 presents a problem in \emph{analysis}, which is circumvented by the \emph{algebraic} definition of h-summability.
Theorem \ref{mainthmintro}
stated in \S1.3 resolves the \emph{analytic}  problem.  Theorem \ref{mainthm} from which Theorem \ref{mainthmintro} is deduced is the main result of the paper. The Poisson
formula plays a central role in this development.

The intrinsic meaning of an infinite sum is our primary goal, not analytic continuation. Nevertheless,  the naturality of the summability definitions yields some applications to analytic continuation when the terms of the series are functions of a complex variable. A few such statements of analytic continuation have been listed in \S1.4.  The resulting meromorphic
functions do not possess a functional equation in general.

The last subsection \S1.5 of the introduction lays out the plan of the paper.
 
\subsection{t-summability}

We begin with the definition of t-summability of a series indexed by a commutative discrete group $\Gamma$.

Given a function
 $a:\Gamma\to\C$, we want to define 
 $\underset{\gamma\in\Gamma}{\Sigma}a(\gamma)\in\C$ in a translation-invariant manner.  Precisely we require axioms (i) and (ii) to hold.
 \\  \textbf{Axiom (i)}: The collection of $a:\G\to\C$ for which  $\underset{\gamma\in\Gamma}{\Sigma}a(\gamma)$ is defined is a complex linear subspace of $\C^{\G}$,
 and $a\mapsto  \underset{\gamma\in\Gamma}{\Sigma}a(\gamma)$ is a complex-linear functional defined on that subspace.

For $\g'\in\G$, define  $a_{\g'}:\G\to\C$ by $a_{\g'}(\g):=a(\g+\g')$ for all $\g\in\G$. 
\\ \textbf{Axiom (ii)}\footnote{\label{1} 
 Condition (C) in \S1.3 of Hardy's  book ``Divergent Series''  can be thought of Axiom (ii) for $\G=\Z$ and $\g'=1$;  Hardy's (A) and (B) are Axiom (i).}: If  $\underset{\gamma\in\Gamma}{\Sigma}a(\gamma)$ is defined,
 then $\underset{\gamma\in\Gamma}{\Sigma}a_{\g'}(\gamma)$ is defined and equals $\underset{\gamma\in\Gamma}{\Sigma}a(\gamma)$.
 
 In the presence of Axiom (i), one checks that Axiom (ii) implies the identity
\begin{equation} (\underset{\gamma\in\Gamma}{\Sigma}c(\gamma))(\underset{\gamma\in\Gamma}{\Sigma}a(\gamma))=\underset{\gamma\in\Gamma}{\Sigma}c*a(\gamma)
\end{equation}
for every finitely supported $c:\Gamma\to\C$, where $c*a$ denotes the convolution of $c$ with $a$.
 Now the above formula is valid when the series $\underset{\gamma\in\Gamma}{\Sigma}a(\gamma)$ is absolutely convergent.  We are thus led to the following definition: the series $\underset{\gamma\in\Gamma}{\Sigma}a(\gamma)$ is translation-summable, or simply t-summable (or t-convergent), 
  if there is a finitely supported $c:\Gamma\to\C$ 
  so that $0\neq\underset{\gamma\in\Gamma}{\Sigma}c(\gamma)$ and
 the series  $\underset{\gamma\in\Gamma}{\Sigma}c*a(\gamma)$ is absolutely convergent. The
 sum of the series is then defined by the formula below:
\begin{equation}
\underset{\gamma\in\Gamma}{\Sigma}a(\gamma):=
 (\underset{\gamma\in\Gamma}{\Sigma}c(\gamma))^{-1} \underset{\gamma\in\Gamma}{\Sigma}c*a(\gamma).
 \end{equation}
\noindent 
\emph{The commutativity hypothesis 
is crucial. }
The commutativity of $\Gamma$ ensures the validity of statements (a) and (b) that follow.
\\(a) The left hand side of equation (2)
 is independent of the choice of $c$.
 \\(b) The space of t-summable functions $a:\G\to\C$ and the map $a\mapsto \underset{\gamma\in\Gamma}{\Sigma}a(\gamma)$  satisfy Axioms (i) and (ii).

 In particular,
if the series $\underset{\gamma\in\Gamma}{\Sigma}a(\gamma)$ is t-summable, then
the series $\underset{\gamma\in\Gamma}{\Sigma}c*a(\gamma)$ is also t-summable,  for every finitely supported 
function $c:\Gamma\to\C$, and equation (1) holds.

The basic example  of a t-summable series for $\Gamma=\Z^k$ that we are concerned with follows.
\\Let  $\T=\{z\in\C:|z|=1\}$.
  \\
 For $n=(n_1,...,n_k)\in\Z^k$ and $z=(z_1,...,z_k)\in \T^k$, we put $z^n=z_1^{n_1}z_2^{n_2}...z_k^{n_k}$.  
\\Given a subset $M\subset \Z^k$ and  $a:M\to\C$, we define $a':\Z^k\to\C$ by
\[a'(n)=a(n)\mbox{ if }n\in M\mbox{ and }a'(n)=0\mbox{ if }n\notin M.\]
When
 $\sum{n\in\Z^k}a'(n)$ is t-summable, we will simply say that $\sum{n\in M}a(n)$
  is t-summable.
  
A function 
$f:\R^k\setminus\{0\}\to\C$ is \emph{homogeneous of type} $(s,\epsilon)\in\C\times\{\pm 1\}$ if 
\begin{equation}\label{homodfn}
f(tx)=t^sf(x)\mbox{ and }f(-x)=\epsilon f(x)\mbox{ for all }t>0,0\neq x\in\R^k
\end{equation}
The theorem that follows is easy. It has been stated because the function $F_f(x,z)$ that emerges, especially its behavior for $x$ in a neighborhood of zero, and
$z=(z_1,z_2,...,z_k)$ in a neighborhood of $(1,1,...,1)\in\T^k$, is the focus of a good part of the paper, specifically in sections 1.2,1.3, 5 and 6.
The remarks following the statement of the theorem explain the situation.
\begin{thm}\label{tcon1prop} Let $f:\R^k\setminus\{0\}\to\C$ be a \df function homogeneous
of type $(-s,\epsilon)$. Then the series 
\begin{equation}\label{tcon1}
F_f(x,z)=\underset{0\neq n\in\Z^k}{\Sigma}f(x+n)z^n\mbox{ is t-summable }
\forall x\in D,\forall z\in\T^k\setminus\{(1,1,...,1)\}
\end{equation}
where $D=\R^k\setminus(\Z^k\setminus\{0\})=(\R^k\setminus\Z^k)\cup\{0\}$.  Furthermore the
sum of this series gives a \df function on the domain $D\times(\T^k\setminus\{1\})$, where $1\in\T^k$ is the multiplicative identity 
$(1,1,...,1)$ of the group $\T^k$.
\end{thm}
The first assertion of the theorem is that the series $F_f(x,z)$  is t-summable. This is implied by the 
absolute convergence of the series $G_{i,N}(x,z):=(1-z_i)^NF_f(x,z)$ whenever $N>>0$. Once this has been checked,  t-summability defines
\\$F_f(x,z):=(1-z_i)^{-N}G_{i,N}(x,z)$ under the assumption that $z_i\neq 1$.  This procedure therefore defines $F_f(x,z)$ when 
$z=(z_1,...,z_k)\in\T^k$ is different from $(1,1,...,1)$. The absolute convergence of $G_{i,N}(x,z)$ relies only on the fundamental theorem of calculus and the
finiteness of the integral $\int_{\R^k}(1+\|x\|)^{-k-1}$. That the sum is \df uses a little more from a first course in Analysis (see Corollary \ref{elem}.)  

Consider the special case $k=1$ and $f(x)=|x|^{-s}$ whenever $x$ is a nonzero real number.  The theorem defines
\[\Sigma_{n=1}^{\infty}|x+n|^{-s}z^n+\Sigma_{n=1}^{\infty}|x-n|^{-s}z^{-n}
\]
as a \df function on the domain $\{(x,z)\in\R\times\T: |x|<1\text{ and }z\neq 1\}$. This is true for every $s\in\C$. Consider next  $z=1$. The series $F(x,1)$  is t-summable if and only if it is absolutely convergent. In particular, its sum is not defined for $\t{Re}(s)\leq 1$ by t-summability.
 For $s\neq 1$, an algebraic definition of $F_f(0,1)$ is given in \S 1.2. 


 
\subsection{h-summability}

With $f$ and $F_f$ as in Theorem \ref{tcon1prop} , 
it would be natural to obtain the `value' of $\sum{0\neq n\in\Z^k}f(n)$  by evaluating
the limit $\underset{z\to 1}{\lim}F_f(0,z)$, but this limit does not exist in general. 
However this problem can be resolved algebraically in the following manner.

  Let $m>1$ be a natural
number. Let $\T^k(m)=\{z\in\T^k:z^m=1\}$.  With $f,F_f$ as is in (\ref{tcon1})
, we have the identity
\[\sum{\lambda\in\T^k(m)}F_f(0,\lambda.z)=m^{k-s}F_f(0,z^m)\,\forall z\in\T^k\mbox{ with }z^m\neq 1\]
We rewrite the above as
\[F_f(0,z)-m^{k-s}F_f(0,z^m)=-\sum{1\neq \lambda\in\T^k(m)}F_f(0,\lambda.z).\]
The term on the right is defined for all $z$ in a neighborhood of $1\in\T^k$. This suggests that
the value of $F_f(0,1)$ can be `forced' by setting $z=1$ in the above formula
when $s\neq k$ by choosing $m\in\N$ so that $m^{s-k}\neq 1$: 

\begin{equation}\label{hintro}
\sum{0\neq n\in\Z^k}f(n)=-(1-m^{k-s})\inverse  \sum{1\neq \lambda\in\T^k(m)}F_f(0,\lambda.)
\end{equation}
The term on the right in equation (\ref{hintro}) is checked be independent
of the choice of $m$. 
The procedure above leads to the definition of h-summability (or h-convergence) of  a series . It is 
in fact an iteration of the method that  defines t-summability. The precise definition
is given in \S2.3. 
The series $\sum{0\neq n\in\Z^k}f(n)$
is h-summable when $(s,\epsilon)\neq (k,1)$ and its sum is given by (\ref{hintro}) when 
$s\neq k$.  When $\epsilon=-1$, the sum of the series is zero.  
 


\subsection{The Poisson formula}  We continue with the $f$ and $F_f$ of Theorem \ref{tcon1prop}. The relation between the sum of
the h-summable series $\sum{0\neq n\in\Z^k}f(n)$ and the
behaviour of $F_f(0,z)$ as $z\in\T$
approaches 1, is given in   Corollary \ref{corintro}. This is obtained from 
Thm. \ref{mainthmintro}, which in turn is deduced from Theorem \ref{mainthm}, the main result of this paper.

We set up the notation for Fourier transforms and the Poisson formula.

\begin{equation}\label{psi-intro}
\psi(x,y)=\exp 2\pi\sqrt{-1}(x_1y_1+x_2y_2+...+x_ky_k)
\end{equation}
for $x=(x_1,...,x_k),y=(y_1,...,y_k)\in\R^k$.
We have the Fourier transform 
\begin{equation}\label{Fintro}
\wh{h}(y)=\int_{\R^k}h(x)\psi(x,y)dx_1dx_2...dx_k\mbox{ for } h\in\abss(\R^k)
\end{equation}
The Poisson formula asserts 
\begin{equation} \Sigma_{n\in\Z^k}h(n)=\Sigma_{n\in\Z^k}\wh{h}(n)
\end{equation}
under certain integrability and regularity assumptions on $h$ (see for instance \S2,Chapter VII of Weil's book \cite{BNT}).  Subtracting $h(0)+\wh{h}(0)$
from both sides, we get a modified  form of the Poisson formula\footnote{\label{2}J-F Burnol's paper \cite{BP} 
has this form of the Poisson formula on p.798; on p.799, he refers  to the work of A.Connes\cite{C} and that of Muntz explained by Titchmarsh
in \S2.11\cite{T}. }

\begin{equation}\label{Poissonmod}
-\wh{h}(0)+\Sigma_{n\in\Z^k\setminus\{0\}}h(n)=-h(0)+\Sigma_{n\in\Z^k\setminus\{0\}}\wh{h}(n)
\end{equation}
that is valid for a wider class of $h$.
\\ \\
We return now to  our \df function $f:\R^k\setminus\{0\}\to\C$ which is homogeneous of type $(-s,\epsilon)$.
It is well known that $f$  extends uniquely as a homogeneous tempered distribution 
(see \cite{H},vol.1,chapter 7) as long as  
\begin{equation}\label{assump}
 (-s,\epsilon)\notin\{(e+k,(-1)^e):e=0,1,2,...\}. 
\end{equation}  

 Its Fourier transform $\wh{f}$ is then \df on $\{0\neq y\in\R^k\}$ and is homogeneous  of type 
 $(-(k-s),\epsilon)$.  
 
 The assumption (\ref{assump}) is equivalent to the assertion that there are no distributions supported at the origin with the same 
 homogeneity degree. The Fourier transform of distributions supported at the origin are polynomials.  Theorem \ref{mainthmintro} requires that we exclude the homogeneity
 degree of  polynomials and their Fourier transforms.
 
We will write
\begin{equation}\label{expt}
\E_k:\R^k\to\T^k\mbox{ for }(x_1,...,x_k)\mapsto (\exp 2\pi\sqrt{-1}x_1,...,\exp 2\pi\sqrt{-1}x_k)
\end{equation}
Theorem \ref{tcon1prop} defines the function $F_f(x,z)$ as the sum of a t-summable series for all $1\neq z\in\T^k$ and all $x\in D$ where 
$D=\R^k\setminus(\Z^k\setminus\{0\})$. Thus  $F_f(x,\E_ky)$ is defined for all $y\in\R^k\setminus\Z^k=D\setminus\{0\}$. We therefore have
the function
\begin{equation}\label{LHS}
-\psi(x,y)\wh{f}(y)+F_f(x,\E_ky)\text{ defined when }(x,y)\in D\times(D\setminus\{0\})
\end{equation}
and the same consideration applied to $\wh{f}$ yields
\begin{equation}\label{RHS}
-f(x)+\psi(-x,y)F_{\wh{f}}(y,-\E_kx)\text{ defined when }(x,y)\in (D\setminus\{0\})\times D
\end{equation}
The Poisson formula as presented in equation (\ref{Poissonmod}), when applied to the function $h$ 
given by  $h(w)=f(w+x)\psi(w,y)$ yields the equality of the functions defined in (\ref{LHS}) and (\ref{RHS}) on the region $(D\setminus\{0\})\times(D\setminus\{0\})$.
However, the absence of absolute convergence, and the singularity at $0\in\R^k$, of $f$ or $\wh{f}$ (or both)
presents a problem that has to be overcome. This statement for $\R$ (rather than $\R^k$) is essentially the functional equation of the 
Lerch zeta function--see equations (1.4), page 161 and (2.4), page 163 of Apostol's paper \cite{A1}.
The desired equality, now on $\R^k$, extends the domain of the functions given in (\ref{LHS}) and (\ref{RHS}) to $D\times D\setminus\{(0,0)\}$. 
Part (1) of the theorem below goes further and extends the domain to $D\times D$.

 \begin{thm}\label{mainthmintro} Let $f:\R^k\setminus\{0\}\to\C$ be \df and homogeneous of type $(-s,\epsilon)$. Assume that both $(-s,\epsilon)$ and
 $(-(k-s),\epsilon)$ satisfy (\ref{assump}). Then there is a \df function \linebreak
 $F_f^{reg}:D\times D\to\C$ with the following properties
 
 \begin{enumerate}
\item $F_f^{reg}$ agrees with the functions given in (\ref{LHS}) and (\ref{RHS}) on their respective domains,
\item $F_f^{reg}(0,0)$ equals $\sum{0\neq n\in\Z^k}f(n)$ as defined by equation 
(\ref{hintro})
\item $ \sum{0\neq n\in\Z^k}f(n)=\sum{0\neq n\in\Z^k}\wh{f}(n) $
\end{enumerate}

 \end{thm}
 
Part (3) follows from interchanging the roles of $f$ and $\wh{f}$. When $k=1$, part (3 ) is simply the functional equation of the Riemann zeta function. 

Part (1) of the theorem applied to the function in (\ref{LHS}) with $x=0$ combines with part (2) of the theorem to give: 
\begin{cor}\label{corintro}
$y\mapsto F_f(0,\E_ky)-\wh{f}(y)$ defined for $0\neq y\in D$ extends to
 a \df function on $D$ whose value at zero is given by $\sum{0\neq n\in\Z^k}
f(n)$.

\end{cor}
We think of the above corollary as the analytic significance of the algebraically
defined sum.
  
  The
environment of Theorem \ref{mainthmintro} for $k=1$
 has been subjected to extensive study.
Our functions $F_f(x,\exp 2\pi\sqrt{-1})y)$ are
closely related to the Lerch zeta function
\[\phi(y,x,s)=\overset{\infty}{\sum{n=0}}(x+n)^{-s}\exp 2\pi\sqrt{-1}ny
\]
whose properties have been investigated in recent times by
J.Lagarias and W.Li. The reader can gather its history from 
their paper \cite{LL}. The functional equation of the Lerch zeta function was proved by Lerch in
1887. Another proof of this theorem is given by 
Apostol \cite{A1} where he follows Riemann's
method. As is well known, this method begins by 
expressing  $\int_0^{\infty}$ as $\int_0^1+\int^{\infty}_1$, then converts
 $\int_0^1$ to $\int_1^{\infty}$ by a change of variables, and proceeds. 
We do not
follow this method, and rely
instead on conventional methods in distributions, namely 
the use of cut-off functions.
A proof of the functional equation of the zeta function that also relies
on distributions, similar but not identical to ours, is due
to S.Miller and W.Schmid in a paper \cite{MS} that  
has far wider goals
and applications.  The paper \cite{GL} of P.Gerardin and W.Li interprets
the functional equation of the zeta function as an equality of tempered
distributions. The recent paper \cite{B} by Burnol also studies
 the Poisson
formula for tempered distributions. 

\subsection{Analytic continuation}\noindent
The complex number `$s$' that has appeared so far has been a constant. In this subsection, `$s$' is a variable. Sums of t-summable and h-summable series 
turn out to be meromorphic functions of $s$ when the terms of the series are holomorphic functions of $s$. We give some instances.
\\  \\\textbf{Example t1.} Let $f:\R^k\setminus\{0\}\to(0,\infty)$ be a \df function such that $f(tx)=tf(x)$ for all $t>0$ and for all $x\in\R^k\setminus\{0\}$
\\ \emph{Then the series $\Sigma_{n\in\Z^k}'z^nf(n)^s$ is t-summable for all $1\neq z\in\T^k$, and all $s\in\C$. The sum of the series is 
 \df$\,$
   in $z$ and holomorphic in $s$. } 
\\This is proved in Corollary \ref{corhomholt}

 Example t1 contains the two ancient examples of analytic continuation stated below.
\\ \textbf{Example t2.}  \emph{The series $\Sigma_{n=1}^{\infty}z^nn^{-s}$ is t-summable for all $1\neq z\in\T$ and for all $s\in\C$. The sum of the series is \df$\,$  in $z$ and holomorphic in $s$.}

In particular,  the familiar series $\Sigma_{n=1}^{\infty}(-1)^{n}n^{-s}$ is t-summable for all $s\in\C$, and its  sum is holomorphic in $s$. Our
t-summability method here is no different from the finite difference method employed by J.Sondow (see \cite{S}) to give an elementary proof of this  well-known statement.
\\ \textbf{Example t3.} The Dirichlet's L-series $L(s,\chi)=\Sigma_{n=1}^{\infty}\chi(n)n^{-s}$ 
is t-summable for every \emph{nontrivial} Dirichlet
character $\chi$, and its sum is an entire function of $s$.
\\This is deduced from linear combinations of Example t1. \\The series that defines the Riemann zeta function however is t-summable if and only if is absolutely
summable. Nevertheless it is h-summable when $s\neq 1$.
\\ \\ \textbf{Example h1.}  \emph{ With $f$ as in Example t1, the series $\Sigma'_{n\in\Z^k} f(n)^{-s}$ is h-summable for all $k\neq s\in\C$. 
The sum of this series gives a holomorphic function on the region $k\neq s\in\C$. At $s=k$,
this function has a simple pole, or is holomorphic.} 

This statement is contained in Proposition \ref{prophomholh}.

Taking $f=\sqrt{Q}$ where $Q$ is a positive definite quadratic form, we obtain the analyticity of the Epstein zeta function.  But unlike the
situation for the Epstein zeta function, there is no functional equation in  general. The Fourier transform of $f^{-s}$ is not $g^{s-k}$ times a function of $s$.
Thus the proof for example h1 is structurally different 
from the classical proof of the analytic continuation of the Epstein zeta function which relies on the Poisson formula.

Homogeneity in Example h1 can be replaced by suitable bounds on partial derivatives.
The example below illustrates this point. We first set up the required notation.
\\Let $P=P_0+P_1+...+P_d$, where the $P_i$ are homogeneous polynomials of degree $i$
on $\R^k$. Assume that $P_d(x)>0$ for all $0\neq x\in\R^k$. We then have
a finite subset $F\subset\Z^k$ such that $P(n)>0$ for all $n\in\Z^k,n\notin F$.
We denote by $\Sigma'$ the sum taken over $\Z^k\setminus F$. 
\\ \textbf{Example h2.} \emph{ With assumptions and notation as stated, the the series
$\Sigma'P(n)^{-s}$ is h-summable for all $s\in\C,s\notin E$, where
$E=\{\frac{k-j}{d}:j\in 2\Z,j\geq 0\}$. The sum of this series, denoted
by $G(s)$, is holomorphic on $\C-E$, and has simple poles at worst on $E$. For $k$ odd and for an integer $s\leq 0$, 
the finite
sum $\underset{n\in F}{\Sigma}P(n)^{-s}$ equals $-G(s)$}.

Example h2 is contained in Theorem \ref{special}. The last assertion  gives a generalization of the values of the Riemann
zeta function at $0,-2,-4,...$. 

\subsection{}The plan of the paper is as follows. The next section has the definitions
of t-summability and h-summability. Sections 3 and 4 are concerned with t-summable series and h-summable series respectively. Sufficient conditions for testing
summability are given. It is also shown the sum of a series is well behaved with respect to parameters. Theorem \ref{tcon1prop} and the analytic continuation results
mentioned in \S1.4 are proved in these sections. 

Section 5 is essentially a review of the Poisson
formula for distributions based on ideas borrowed from \cite{acta},\cite{LV}
 and \cite{M}, followed by some simple analysis of the singularities.
This discussion is applied to a class of distributions $\HH(\R^k)$
defined in section 6. The main result there (and of the paper) is Theorem \ref{mainthm}, which is valid for a considerably wider class of functions than those that 
appear in Theorem. \ref{mainthmintro}. 
Homogeneity is not required-- every extension of $f$ as a distribution on $\R^k$ `forces' a definition of $\Sigma'_{n\in\Z^k} f(n)$.
This section closes with a definition of t-integrability of distributions\footnote{\label{3}Akshay Venkatesh points out that
 Kudla and Rallis have more sophisticated versions and applications of translation-invariance in \cite{KR}}
on $\R^k$. 

The
last section discusses the more general definition of the sum of a
series indexed by a set $X$ which a group $G$ acts.
The $H$-summability definition of \S7 
is more stringent than h-summability. It has the advantage of getting
 rid of the anomalous behaviour of h-summability in part (2) of Theorem \ref{special} and  Example \ref{insecure}. The disadvantage of $H$-summability is that  the
 series that defines the Riemann zeta function  is \emph{not} $H$-summable  for $s=-1,-3,-5,...$, whereas this series is h-summable for all $s\neq 1$.

A word of apology regarding notation. The
Euclidean space figuring in sections 3 and 4 is invariably $\R^k$. Sections
5 and 6 have finite dimensional real vector spaces $X$ and $X'$, and the
 `$k$'s of those sections are variable non-negative integers. Finally,
in sections 2 and 7, $X$ denotes a set equipped with the action of a group $G$. 

\section{definition of $t$-summability and $h$-summability}
\subsection{The canonical extension}\label{canext} Given 
\\(a)a commutative ring $A$, a field $k$, and   
a ring homomorphism $\epsilon:A\to k$ where $k$ is a field,
\\(b) an inclusion $M\hookrightarrow N$ of left $A$-modules, and 
\\(c) a left $A$-module homomorphism $I:M\to k$, and

we define its \emph{canonical extension} $I_c:M_c\to k$  as follows.

Let $S=\{a\in A:\epsilon(a)\neq 0\}$. This is a multiplicative subset of $A$.
\\We obtain $S^{-1}I:S^{-1}M\to k$. Let  $i:N\to S^{-1}N$ denote 
the homomorphism $n\mapsto\frac{n}{1}$ for all $n\in N$. Note that
$S^{-1}M$ is contained in $S^{-1}N$. We define $M_c$ by
\begin{equation}\label{canextM}
M_c=\{n\in N:i(n)\in S^{-1}M\}=\{n\in N:\exists s\in S \mbox{ so that }sn\in M\}
\end{equation}
Now $i:N\to S^{-1}N$ restricts to $j:M_c\to S^{-1}M$. Denoting by $I_c$
the composite
\begin{equation}\label{canextI}
M_c\xr{j}S^{-1}M\xr{S^{-1}I}k
\end{equation}
we obtain $I_c:M_c\to k$ with the property $I=I_c|_M$. 

\emph{Then $I_c:M_c\to k$ will be referred to as the canonical extension of the
$A$-module homomorphism $I:M\to k$ }.

\subsection{t-summability} 
\noindent
\\The canonical extension is now applied to the following situation:
\\(a)  $A=\C[\G]$, the group-algebra of a discrete abelian group $\G$, the field $k$ equals $\C$, and the augmentation $\epsilon:A\to k$ takes
$\Sigma_{\g\in\G}c_{\g}\g\in A$ to $\Sigma_{\g\in\G}c_{\g}$.
\\(b) $M=\abss(\G)$, the collection $f:\G\to\C$ satisfying  $\underset{\g\in\G}{\Sigma}|f(\g)|<\infty$. 
\\$N=\C^{\G}$ is the space of all $\C$-valued functions on $\G$.  The collection of finitely supported $f\in N$ is $A$. Therefore 
we have $A\subset M\subset N$. 
\\The convolution $(f_1*f_2)(\g)=\Sigma_{\g'\in\G}f_1(\g')f_2(\g'^{-1}\g)$ is defined when $f_1\in A$ and $f_2\in N$. This gives $N$ the
structure of a $A$-module. 
\\If $f_1,f_2\in M$, then $f_1*f_2$ is defined and belongs to $M$. This gives $M$ the structure of a $\C$-algebra, and $A$ is then 
a $\C$-subalgebra of $M$. In particular, $M$ is a $A$-submodule of $N$.
\\(c) $I:M\to\C$  is given by $I(f)=\Sigma_{\g\in\G}f(\g)$. Now $I$ is a ring homomorphism, and $I|_A=\epsilon$. It follows that 
$I$ is a $A$-module homomorphism.
\\ \emph{The  enlarged domain $M_c$  will be denoted by $\ts(\G)$. The canonical extension $I_c$ will be denoted by $\Sigma_{\t{t}}:\ts(\G)\to\C$, the notation 
$\Sigma_{\t{t}}$ denoting ``t-summation''.  A function $f:\G\to\C$ is t-summable if and only if $f$ belongs to $\ts(\G)$.} 

Let $f\in \ts(\G)$, let $\g'\in\G$, and consider  the $\g'$-translate $f'$ of $f$ given by $f'(\g)=f(\g+\g')$. Then $f'$ also belongs to $\ts(\G)$,
and $\Sigma_{\t{t}}f=\Sigma_{\t{t}}f'$.

This definition of t-summability is no different from the one given in the
introduction. It is important that $\G$ is an abelian group.
\begin{lem}\label{tstab} Let $\G_1$ be a subgroup of an abelian group $\G_2$. Given a function $a:\G_1\to\C$, we extend it by zero to obtain $a':\G_2\to\C$. Then 
$a'$ is t-summable if and only if $a$ is t-summable. 
\end{lem} 
\begin{proof}
The `if' part is straight from the definition. To see the converse, assume that $a'$ is t-summable. We have $s\in\C[\G_2]$ with $\epsilon(s)\neq 0$
such that the convolution $s*a'$ belongs to $\abss(\G_2)$.  Both properties of $s$ are preserved when $s$ is replaced by a translate of $s$. Thus we may assume 
that $\epsilon(s|\G_1)$ is nonzero. The restriction of $s*a'$ to $\G_1$ is $(s|\G_1)*a$. The latter is therefore absolutely summable. It follows that $a$ is t-summable.
\end{proof}
\subsection{h-summability} Here, $R$ is a commutative ring, and $\Gamma$ 
is a $R$-module. 

In order to set up the data for the canonical extension, we observe first that $R\units$, the
group of units of $R$, operates on $\G$ and therefore on the space of functions $\C^{\G}$
as well. It is a simple
consequence of the definition of t-summability that 
\\(i) $\ts(\G)\subset\C^{\G}$ is a $R\units$-submodule, and (ii) $\Sigma_{\t{t}}$ is a $R\units$-invariant functional. 
We now have:
\\(a) $A=\C[R\units]$, $k=\C$, and $\epsilon:A\to\C$ is the augmentation of the group-algebra,
\\(b) $M=\ts(\G)$ and $N=\C^{\G}$,
\\(c)  $I:M\to\C$ is given by $\Sigma_{\t{t}}:\ts(\G)\to\C$.
 
 The canonical extension $I_c:M_c\to\C$ associated to the data specified 
 will be denoted by 
$\Sigma_{\t{h}}:\hs(\G)\to\C$. A function $f:\G\to\C$ is said to be h-summable if and only if $f\in\hs(\G)$.

We summarise the properties of the above construction in the proposition below.

\begin{prop}\label{summary}
Let $\Gamma$ be a $R$-module, where $R$ is a commutative ring. Then
\\(a)  The spaces $\abss(\Gamma)\subset \ts(\Gamma)\subset 
\hs(\Gamma)\subset \C^{\Gamma}$ are all stable under the natural action of $R\units$. 
\\(b) $\Sigma_{\t{h}}:\hs(\Gamma)\to\C$ is a $R\units$-invariant
linear functional.
\\ (c) The space $\ts(\Gamma)$ is stable under translation by $\Gamma$;
the restriction of $\Sigma_{\t{h}}$ to $\ts(\Gamma)$ is a $\Gamma$-invariant linear functional.
\end{prop}
In part (c) of the Proposition, the restriction of $\Sigma_{\t{h}}$ to $\ts(\Gamma)$ is of course $\Sigma_{\t{t}}$.
\subsection{Instability of h-summability}
\begin{ex}\label{instab1}
The space of h-summable functions on $\G$ need not be
stable under translations by $\G$. To see this, let
$R=\G$ be a subring of $\C$. 
\\ \emph{Claim 1}:
$f(x)=x$ for all $x\in\G$ is h-summable. 
\\Proof: because $\frac{1}{a}f(ax)-f(x)=0$ for all $a\in R\units$, we see that
$s.f=0$ where $s=\frac{1}{a}[a]-[1]\in\C[R\units]$.  Now $\epsilon(s)=\frac{1}{a}-1$ is nonzero if $a\neq 1$. Taking $a=(-1)$ we see that $f$ is h-summable.
\\ \emph{Claim 2}: the constant function $v(x)=1$ for all $x\in\G$ is not h-summable. 
\\Proof: If $v$ is h-summable, there is some $s\in\C[R\units]$ with $\epsilon(s)\neq 0$
such that $s.v$ is t-summable. This in turn implies that there is some $s'\in\C[\G]$ with $\epsilon(s')\neq 0$ such that $s'.s.v$ is absolutely summable. But $s'.s.v$ is the 
nonzero constant function $\epsilon(s')\epsilon(s)$, and so we get a contradiction.

The claims combine to show that $f$ is h-summable, but its translate $f(x+1)=f(x)+v(x)$ is not h-summable.
\end{ex}

\begin{rem}\label{instab} Now let $\G_2$ be a $R_2$-submodule, let $R_1$ be a subring of $R_2$, and let $\G_1$ be a $R_1$-submodule of $\G_2$. Let
$a':\G_2\to\C$ denote the extension by zero of $a:\G_1\to\C$.
The h-stability of $a'$ does not imply the h-stability of $a$. Take, for example, $R_1=\G_1=\Z$ and $R_2=\G_2=\Q$,
and let $a$ be the constant function $1$ on $\Z$.
It has been remarked in Example \ref{instab1} that $a$ is not h-summable. But it will be seen later (in Proposition \ref{hpol})
that $a'$ is h-summable, and also that $\Sigma_{\t{h}}a'=0$.  

The fact that the space of h-summable functions grows when $(\G_1,R_1)$
is replaced by $(\G_2,R_2)$ is to our advantage, because it defines the sum
of more series in a canonical manner. The example below shows however that
this sum has to be treated with caution.

\end{rem}

\begin{ex}\label{insecure} We observe here that even when the translates of a h-summable $f$ are h-summable, the sum  is not translation invariant.

Take $\R=\G=\Q$. Consider the function $f:\Q\to\C$ which takes the value 1 on the set of natural numbers, and is zero elsewhere. By Prop. \ref{hpol}, we see that
$f(x)+f(-x)$ is h-summable, and its sum is (-1). Because $s=[-1]+[1]\in\C[\Q]$ satisfies $\epsilon(s)=2$, we deduce that $f$ is h-summable, and $\Sigma_{\t{h}}(f)=\frac{-1}{2}$. 
Now $g(x)=f(x+1)$ is h-summable because  $f-g$  has finite support. But 
$\Sigma_{\t{h}}(f)=\Sigma(f-g)+\Sigma_{\t{h}}(g)=-1+\Sigma_{\t{h}}(g)$.  Thus $\Sigma_{\t{h}}f(x)=-1+\Sigma_{\t{h}}f(x+1)$.

\end{ex}
\subsection{An inductive definition} We apply the canonical extension repeatedly to 
the following situation (which contains the earlier summability definitions).  

Let $G$ be a group equipped with 
\emph{commutative} subgroups 
 $H_1,H_2,...,H_n$ with the 
  property that $H_j$ is contained in the normaliser of $H_i$ whenever $1\leq i<j\leq n$. It follows
  that 
  \\(i) $G_k=H_kH_{k+1}...H_n$ is a subgroup of $G$, and (ii) $H_k\triangleleft G_k$ 
  for all $k=1,2,...,n$. 
\\Next, let $N$ be a left $\C[G]$-module, let $M\subset N$ be a $\C[G]$-submodule,
and let $I_0:M\to\C$ be a $G$-invariant linear functional.

We define inductively, 
\\(a)$M=M_0\subset M_1\subset ...\subset M_n\subset N$ such that $M_k$ is 
a representation of $G_k$,
\\(b) a $G_k$-invariant linear functional $I_k:M_k\to\C$ such that
the restriction of $I_k$ to $M_{k-1}$ is $I_{k-1}$ for all $1\leq k\leq n$.
\sloppypar
Let $G_0=G$. Assume that
 the $G_k$-submodule $M_k\subset N$ and the $G_k$-invariant
linear functional 
$I_k:M_k\to\C$ have been defined for some $0\leq k<n$. We regard $M_k,N$
as $H_k$-representations, and $I_k$ as a $H_k$-invariant linear functional.
Because $H_k$ is commutative, the canonical extension of $I_k$
is defined, and denoted
by
\linebreak
 $I_{k+1}:M_{k+1}\to\C$.Because $H_k$ is normal in $G_k$, it follows that
$M_{k+1}$ is $G_k$-stable, and $I_{k+1}$ is $G_k$-invariant. The pair
$(M_{k+1},I_{k+1})$ enjoys the same properties when $G_k$ is replaced by
a subgroup of $G_k$. Taking the subgroup to be $G_{k+1}$ 
completes the induction.

Note
that $M_n$ is a $H_n$-representation, and $I_n:M_n\to\C$ is $H_n$-invariant.

In practise, $N=\abss(X)$ and $M=\C^X$, and $I_0=\Sigma=\Sigma_X$.
Here $X$ is a set on which $G$ operates.

h-summability is contained in the case $n=2$.
\section{$\t{t}$-summability}
\noindent
A commutative discrete group $\G$ remains fixed throughout this section.
The ring structure of the group algebra $\C[\G]$ is given by convolution.
\emph{For $d\in\C[\G]$, we will denote the $n$-th power of $d$ in this ring,
 namely $d*d*...*d$ ``$n$ times''  by $d^{(n)}$}.
\\ \\ $\T=\{z\in\C;|z|=1\}$. 
$\G^*=\t{Hom}(\G,\T)$ is the Pontrjagin dual of a commutative group $\G$. Given functions 
$a,b:\G\to\C$, pointwise multiplication gives $a.b:\G\to\C$ and convolution (when defined)
gives another function $a*b:\G\to\C$.  
Given $a:\G\to\C$, we shall frequently say that
the series $\sum{\g\in\G}a(\g)$ is t-summable, or simply that $a$ is t-summable, when $a\in\ts(\G)$.
\emph{The sum of the series, denoted by $\Sigma_{\t{t}}a$ in the previous section, will now be referred to  
simply as $\Sigma a$}. The lemma
below demonstrates that t-summability behaves well.
   \begin{lem}\label{obvious}
    Let $a:\Gamma\to\C$ be a function and let $\chi_0\in\Gamma^*$. 
Assume that the series $\Sigma(a.\chi_0)$ is t-summable. 
Then the series $\Sigma(a.\chi)$ is t-summable for all $\chi$ in some neighborhood $U$ of $\chi_0$ in $\Gamma^*$.
Furthermore,  $\chi\mapsto \Sigma(a.\chi)$ gives  a continuous function on $U$.
 
 \end{lem}
 \begin{proof} The t-summability of $a.\chi_0$ says there is some 
$c\in\C[\Gamma]$, thus $c$ is a function on $\Gamma$ with finite support,
so that
 \[\mbox{(i) } \Sigma c\neq 0\mbox{  and (ii) } b=c*(a.\chi_0)\in\abss(\Gamma).\]
 
 Let $h(\chi)=\Sigma (c.\chi)$  and $g(\chi)=\Sigma (b.\chi)$ for all $\chi\in \Gamma^*$.
 Both $g$ and $h$  are continuous functions on $\G^*$. Let $V$ be the open subset
 of $\Gamma^*$  given by  $h(\chi)\neq 0$. By assumption (i) above, $1\in V$. 
 Now convolution commutes with pointwise multiplication by a character. 
 From this we see that $a.\chi_0.\chi$ is t-summable when $\chi\in V$ and
 also that $\Sigma(a.\chi_0.\chi)=g(\chi)/h(\chi)$ for all $\chi\in V$.
  
 \end{proof}\subsection{Sufficient conditions for t-summability}

\emph{Recall that the $n$-th power of $d$  in the group-algebra $\C[\G]$ is denoted by $d^{(n)}$ .}
\sloppy
\begin{lem}\label{fd} 
Let $ c=\Sigma_{i=1}^{i=k}(1-\g_i)*(1-\g_i\inverse)\in\C[\G]$. Assume that
$\g_1,\g_2,...,\g_k$ are generators of $\G$. Let $a:\G\to\C$ be a function.  For the statements
below,  (1)$\implies$ (2)$\implies$ (3) 
 \begin{enumerate}
\item  there is a natural number $n$ so that  $(1-\g_i)^{(n)}*a\in\abss(\G)$ for all $i=1,...,k$.
\item  there is a natural number $n$ so that $c^{(n)}*a\in\abss(\G)$
\item $a.\chi$ is t-summable for all $1\neq \chi\in\G^*$.
\end{enumerate}

\end{lem}
\begin{proof} We see that $c^{(n)}$ is in the ideal generated by $\{(1-\g_i)^{(2n)}: i=1,2,...,k\}$
and this proves the first implication.  We note next that $h(\chi)=\Sigma(c.\chi)$ is the sum
of $|1-\chi(\g_i)|^2$ taken over $i=1,2,...,k$, and thus $h(\chi)>0$ if $\chi\neq 1$.  It
follows that $\Sigma c^{(n)}.\chi=h(\chi)^n>0$ for such $\chi$. Because
$(c^{(n)}*a).\chi=(c^{(n)}.\chi)*(a.\chi)$,  we conclude the t-summability of $a.\chi$ when $\chi\neq 1$.

\end{proof}
Lemmas \ref{word} and \ref{tpol} will be appealed to in the 
next section on h-summability; they are irrelevant for the rest of this section.  

\begin{lem}\label{word}
Assume that $a:\G\to\C$ satisfies condition \emph{(1)} of lemma \ref{fd}. 
Let $N\in\N$ and let $p:\G\to\G/N\G$ denote the projection. Let $b:\G/N\G\to\C$ be a function
  satisfying $\Sigma_{\G/N\G}b=0$. 
\\(1) Then the series $\Sigma (b\circ p).a$ is t-summable. 

Let $w_i=1+\g_i+...+\g_i^{(N-1)}$
and let $w=w_1^{(n)}*...*w_k^{(n)}$. 
\\(2) Then $\Sigma w=N^{nk}>0$ and $w*((b\circ p).a)$ is in $\abss(\G)$.  

\end{lem}
\begin{proof} The hypothesis  $\Sigma_{\G/N\G}b=0$ implies that $b$ is a linear combination
of nontrivial characters $\chi$ of $\G/N\G$. Thus (1) follows from the previous lemma. 
It suffices to prove part (2) of the lemma for $b=\chi$ with $\chi$ as above. Now (1) of lemma \ref{fd} implies
that 
\\$(1-\chi(p\g_i)\g_i)^{(n)}*(a.(\chi\circ p))$ is in $\abss(\G)$. Part (2) is proved once it is checked that $w$ is in the ideal in the group algebra
of $\Gamma$ generated by $(1-\chi(p\g_i)\g_i)^{(n)}$ for $i=1,2,...,k$.
 Choose $i$ so that
$\chi(p\g_i)\neq 1$. Now $w_i$  is a multiple of $1-\chi(p\g_i)\g_i$. It follows that both $w_i^{(n)}$ and $w$  are multiples of  $(1-\chi(p\g_i)\g_i)^{(n)}$.
This completes the proof of part (2).

\end{proof}
\begin{lem}\label{tpol}With the assumptions of lemma \ref{fd}, assume there
is some natural number $n$ so that $(1-\g_i)^{(n)}*a=0$ for all $i=1,2,...,k$.
 Then $a.\chi$ is t-summable for all $1\neq \chi\in\t{Hom}(\G,\C^\times)$,
 and in fact $\Sigma a.\chi=0$.

\end{lem}
\begin{proof}This follows from
$0=((1-\g_i)^{(n)}*a).\chi=(1-\chi(\g_i)\g_i)^{(n)}*(a.\chi)$ and
\\$\Sigma(1-\chi(\g_i)\g_i)^{(n)}=(1-\chi(\g_i))^n$, noting that
$1\neq \chi(\g_i)$ for some $i$, once it assumed that $\chi\neq 1$.
\end{proof}
\subsection{t-summability with parameters}
 It will be necessary to state everything with parameters. Our parameter-space is a topological space
 $W$.  The domain of several functions considered is the space $W\times\G$, which inherits its
 product topology. Frequently, $W$ is an open subset of Euclidean space.
 
 We
 collect in lemma \ref{cont} below some basic properties easily
derived from a first course in Analysis
such as \cite{A},\cite{R}. We set up the requisite notation first. 
\\Given $a:\G\to\C$ and a finite subset $S\subset\G$, we set $\|a,S\|_1=\Sigma\{|a(\g)|:\g\notin S\}$
 \\Given a function $a:W\times\G\to\C$, we put 
 \\$a_w(\g)=a(w,\g)$ for all  $w\in W,\g\in\G$ and  define $\|a,S\|_1=\sup\{\|a_w,S\|_1:w\in W\}$. 
 \\ $a:W\times\G\to\C$ is \emph{acceptable} if 
$\inf\{\|a,S\|_1: S\subset\G,\#S<\infty\}=0$.

We assume that $\G$ is discrete and $W$ is a topological space, and then $W\times\G$
is endowed with its product topology.

\begin{lem}\label{cont}
 Assume that $a:W\times\G\to\C$ is acceptable.  For $w\in W$ and $\chi\in\G^*$, let $G_a(w,\chi)$ be the sum of the absolutely convergent series $\Sigma \chi.a_w$.
 \begin{enumerate}
 \item If  $a$ is continuous , then $G_a:W\times\G^*\to\C$ is also continuous. 
 \item If $W\subset\C$ is open and $w\mapsto a(w,\g)$ is holomorphic for every
 $\g\in\G$, then $w\mapsto G_a(w,\chi)$ is holomorphic for every $\chi\in\G^*$.
 \item If $W\subset \R^m$ is open, if $v\in\R^m$, and if the directional derivative $\pd_va$
 of $w\mapsto a(w,\g)$ exists, is acceptable and continuous,  then the directional 
 derivative $\pd_vG_a$ of the function $w\mapsto G_a(w,\chi)$ exists for all $w\in W,\chi\in\G^*$ and
 is in fact given by $\pd_vG_a=G_b$ where $b=\pd_va$.
 
 \end{enumerate}
  
 \end{lem}
 Part (1) of the lemma is the Weierstrass M-test rephrased to suit the situation at hand. Part (2) 
 of the lemma is the basis of analytic continuation in the parameters.
\subsection{t-summability for $\G=\Z^k$}

 For the rest of this section, we take
 $\G=\Z^k$ and $\G^*=\T^k$. The given basis of $\Z^k$, namely $(1,0,0....,0),(0,1,0,...,0),...
 (0,...,0,1)$,  will be denoted $e_1,e_2,...,e_k$. Given $a:\Z^k\to\C$ we put $\Delta_ia=(1-e_i)*a$.
  Thus $\Delta_ia(n)=a(n)-a(n-e_i)$ for all $n\in\Z^k$.  Iterating this operator $m$ times
  we have $(1-e_i)^{(m)}*a=\Delta_i^ma$. We wish to find sufficient conditions that ensure
  $\Delta_i^ma\in\abss(\Z^k)$ for all $1\leq i\leq k$ in order to appeal to lemma \ref{fd}. 
   We do so when 
  $a=f|_{\Z^k}$ and $f:\R^k\to\C$ is a \df function.  We use the same notation $\Delta_i$
  for functions on $\R^k$ as well. Expressing $\Delta_if$ as the integral of $\pd_if$
  and iterating this procedure, we get
  
  \[
\Delta^m_if(x)=\int_0^1...\int_0^1\partial_i^mf(x-(t_1+t_2+...+t_m)e_i)dt_1dt_2...dt_m
\]
which gives 
\begin{equation}\label{ftcc}
|\Delta^m_if(x)|\leq\sup\{|\pd_i^mf(x-te_i)|: 0\leq t\leq m\}.
\end{equation} 
\emph{For the rest of this section, for $g:\R^k\to\C$ we set}
\begin{equation}\label{norm}\|g\|''=\sup\{(1+\|x\|)^{k+1}|g(x)|:x\in\R^k\}
\end{equation} 
$\|x \|^2=\dot{x}{x}$ is the standard Euclidean norm on $\R^k$. We obtain a constant $C(m,k)$
\begin{equation}\label{est1}
\sum{n\in\Z^k}\sup\{|g(n+y)|:y\in\R^k,\|y\|\leq m\}\leq C(m,k)\|g\|''
\end{equation} 
Putting all the above together we see
\begin{equation}\label{ftccc}
\sum{n\in\Z^k}|\Delta_i^mf(n)|\leq C(m,k)\|\pd_i^mf\|''
\end{equation}
So, to apply lemma 3.2, all we require is the finiteness of $\|\pd^m_if\|''$ for all $1\leq i\leq k$ and for some $m$
with $\|\,\|''$ as given in (\ref{norm}).   A  convenient class of functions which satisfies this 
condition  is given below.
\begin{dfn}\label{Ht} A $\C$-valued \df function $f$ defined on the complement of a compact
 $K\subset\R^k$ is $\hinf$ if  there is some $p\in\R$ for which
\begin{equation}\label{Ht1}
\|x\|^r\pd_v^rf(x)\mbox { is }\O(\|x\|^p)\mbox{ as }\|x\|\to\infty,\,\,\forall v\in\R^k,\forall r\geq 0
\end{equation}
where $\pd_v$ denotes the directional derivative.

Let  $f:W\times  (\R^k\setminus K)\to\C$ be a function so that $f_w$ is \df 
 for all $w\in W$,
(where $f_w(x)=f(w,x)$ for all $w\in W,x\in\R^k$). 
\\We say $f$ is \emph{uniformly} $\hinf$ if there
 \begin{equation}\label{Htu}
 \sup\{\|x\|^{r-p}|\pd_v^rf(w,x)|:w\in W,x\in \R^k,x\notin K'\}<\infty,\,\,\forall v\in\R^k,\forall r\geq 0
 \end{equation}
 for some $p\in\R$ and some compact subset $K'$ of $\R^k$ that contains $K$.

Note that if $f:\R^k\to\C$ is \df, (\ref{Ht1}) is equivalent to (\ref{Ht2}) below

\begin{equation}\label{Ht2}
\sup\{(1+\|x\|)^{r-p}|\pd_v^rf(x)|:x\in\R^k\}<\infty\,\,\forall v\in\R^k,\forall r\geq 0
\end{equation}
 
\end{dfn}
We are ready to state and prove the main result of this section. The notation $\Sigma'$
 below stands for the sum taken  over $\{n\in\Z^k,n\notin K\}$. Precisely, the $n$-th
 term of the series is defined to be zero when $n\in K\cap\Z^k$, and the resulting series
 whose terms are indexed by $\Z^k$ is tested for t-summability.
 \begin{thm}\label{thmtcon}
 \noindent
 (A)
 Assume $f:\R^n\setminus K\to\C$ is $\hinf$. Then the series $\sum{n\in\Z^n}'f(n)z^n$
 is t-summable for all $1\neq z\in\T^k$. The sum of the series, denoted by $G_f(z)$,  is continuous on
 the domain $1\neq z\in\T^k$.
\\ (B) Assume that  $f:W\times(\R^k\setminus K)\to\C$ is uniformly $\hinf$. Denote the sum
 of the t-summable series $\sum{n\in\Z^k}f(w,n)z^n$ by $G_f(w,z)$ for all $w\in W,1\neq z\in\T^k$.
 If $f$ is continuous, so is $G_f$. 
 \\(C) With notation and assumptions as in (B), assume in addition that $W\subset \C$ is open, and moreover that
 $w\mapsto f(w,x)$ is holomorphic for every 
 $x\in\R^k\setminus K$. Then $w\mapsto G_f(w,z)$ is a holomorphic function of $w\in W$ 
 for every $1\neq z\in\T^k$.
 \end{thm}

 \begin{proof} We first prove all three parts of the theorem under the assumption that $K$
 is empty.

In (A) we therefore assume that $f$ is defined on all of $\R^k$. The estimate (\ref{Ht2}) is valid
 now.
With $p$ as in that estimate, we choose any $m$ so that $m-p\geq(k+1)$. It follows that
 $\|\pd_i^mf\|''<\infty$ for all $1\leq i\leq k$. The bound (\ref{ftccc})
 then shows that $\Delta_i^mf|_{\Z^k}\in\abss(\Z^k)$ for all $1\leq i\leq k$. The ``(1) implies
  (3)'' of lemma \ref{fd} proves the t-convergence of the given series. By lemma \ref{obvious}
  we see that the sum of this series is continuous.

  For parts (B) and (C), we once again choose $m$ in exactly the same manner, but
   with $p$ as in (\ref{Htu}). We deduce that $\{\|\Delta_i^mf_w|_{\Z^k}\|_1:w\in W,1\leq i\leq k\}$
   is bounded above. By remark \ref{cont}, we see that  $u_i$ defined by
   \[u_i(w,z)=\sum{n\in\Z^k}\Delta^m_if(w,n)z^n\mbox{ for all }w\in W,z\in\T^k
   \]
   is a continuous function (and holomorphic in $w$ under the assumptions of (C)).
 Because  $(1-z_i)^mG_f(w,z)=u_i(w,z)$ for all $1\neq z\in\T^k$ and for all $1\leq i\leq k$,
 we see that the function $G_f(w,z)$ has the same property. This completes the proof of
 the theorem when $K$ is empty.
 
 We now come to the general case. Choose a test function $\phi\in\test(\R^k)$ satisfying
 \[K\subset\subset\{x\in\R^k:\phi(x)=1\}\mbox{ and }\Z^k\cap\t{supp}(\phi)=\Z^k\cap K.
 \]
 Given $f$ as in (B), we note that 
 \[h(w,x)=(1-\phi(x))f(w,x)\mbox{ if }x\notin K\mbox{ and }h(w,x)=0\mbox{ if }x\in K
 \]
 defines $h$ on all of $W\times\R^k$.  Furthermore, if $f$ satisfies the assumption of (C),
 so does $h$.  Parts (B) and (C) have already been proved for $h$. We note that
 \[\sum{n\in\Z^k}h(w,x)z^n=\sum{n\in\Z^k}'f(w,n)z^n.
 \] 
 This finishes the proof of (B) and (C) in general. (A) is the special case of (B) when $W$
 is a point.
  \end{proof}
\begin{rem}  The sum of the series in Theorem~ \ref{thmtcon}(A) is in fact \df at $1\neq z\in\T^k$.
This can be seen by noting that, with $m$ as in the above proof,
  the same argument shows that  $\Delta_i^{m+d}Pf|_{\Z^k}\in\abss(\Z^k)$
   for every polynomial of degree at most $d$. We skip the details. The
statement is contained in  
Lemma \ref{mainthmpre}.  
  \end{rem}
  Theorem \ref{tcon1prop}  as stated in the introduction, is a special case of the corollary below.
  \begin{cor}\label{elem}
 
  Assume $f:\R^k\setminus K\to\C$ is $\hinf$, let $F\subset \Z^k$ be a finite subset, 
  and let  $D=\{x\in\R^k: K-x\subset F\}$. The series $\Sigma\{f(x+n)z^n: n\in\Z^k, n\notin F\}$  is t-summable for all $x\in D,z\in\T^k, z\neq 1$. The sum of this series 
  is a \df function on $D\times(\T^k\setminus\{1\})$.

  \end{cor}\label{corhomt}
\begin{proof} The t-summability assertion follows from part (A) of Theorem 3.7, once it is noted that
translates of $\hinf$ functions are also $\hinf$.  Applying part (B) of that theorem to the function $(x,y)\mapsto f(x+y)$ on the domain $W\times (\R^k\setminus K)$,
where $W$ is a relatively compact subset of $D$, the continuity in the second assertion follows.  Part (3) of Lemma 3.5, combined with the above remark.

\end{proof}
\begin{cor}\label{corhomholt} Assume $f,g:\R^k\setminus\{0\}\to\C$ are both \df
and homogeneous of type $(1,1)$ and $(s_0,\epsilon)$ respectively. Assume 
furthermore that $f(x)$ is a positive real number for all $0\neq x\in\R^k$.
Then the t-convergent sum $\sum{n\in\Z^k}'f(x)^{-s}g(x) z^n$ is a holomorphic function of $s\in\C$
for all $x\in\R^k$ and for all $1\neq z\in\T^k$.

\end{cor}
\begin{proof} To apply Theorem 3.7(C), one has to observe that 
$h(s,x)=f(x)^{-s}g(x)$ is uniformly $\hinf$ on $W\times\{x\in\R^k:\|x\|>r\}$ where
$W$ is a bounded subset of $\C$, and $r$ is any positive real number.
\end{proof}
As remarked in the introduction, this implies the analytic continuation of $L(\chi,s)$
for nontrivial Dirichlet characters $\chi$.

The proofs of t-summability given so far suggest a broader definition,
 that of t-integrability on $\R^k$. This requires the language of distributions, and is given in \ref{tplus}.

\section    {$\t{h}$-summability on $\Q^k$}
\begin{rem}\label{req}
Before beginning on h-summability, we inform the reader that the simple Lemma \ref{tstab} on t-summability with $\G_1=\Z^k$ and $\G_2=\Q^k$ 
will be invoked several times in this section.

\end{rem}
\begin{notn}\label{hnotn}
Recall that h-summability has been defined for $R$-modules $\G$, where $R$ is a 
commutative ring. In this section, $R=\Q$ and $\G=\Q^k$.
 \\ Given $a:\Q^k\to\C$ and a subset $M\subset \Q^k$, we have $\res_Ma:\Q^k\to\C$
 given by \[\res_Ma(x)=a(x) \mbox{ if }x\in M,\mbox { and }\res_Ma(x)=0\mbox{ if }x\notin M.\] 
 For $t\in\Q\units$ and $a:\Q^k\to\C$, we define $\rho(t)a(x)=a(t\inverse x),\forall x\in\Q^k$.
We observe that 
\begin{equation}\label{gset}
\rho(t)\res_M\rho(t)\inverse f=\res_{tM}f\,\,\forall M\subset\Q^k,\forall t\in\Q\units,\forall f:\Q^k\to\C.
\end{equation}
Given $M\subset\Q^k$ and $a:M\to\C$, we extend $a$ by zero and obtain $a':\Q^k \to\C$.
When $\sum{n\in\Q^k}a'(n)$ is h-summable, we will often say that $\sum{n\in M}a(n)$
is h-summable, or even more simply that $a$ is h-summable.
\end{notn} 
\begin{prop}\label{hh} Let $f:\R^k\setminus\{0\}\to\C$ be \df and homogeneous of 
type $(-s,\epsilon)$. Then
\begin{enumerate}
\item
$\sum{0\neq n\in\Z^k}f(n)$ is h-summable
if $(-s,\epsilon)\neq (k,1)$. 

\item Assume $x\notin\Z^k$. Then $\sum{n\in\Z^k}f(x+n)$ is h-summable if 
$(-s,\epsilon)\notin\{(k-i,(-1)^i):i=0,1,2,...\}$. 
\end{enumerate}
\end{prop}
\begin{proof} We define $f(0)=0$ and regard $\R^k$ as the domain of $f$.
 
$\sum{n\in\Z^k}f(n)z^n$ is t-convergent for $z\neq 1$ by  
Corollary \ref{elem}.  Summing over the nontrivial N-torsion points of $\T^k$, we deduce
that 
$N^k\res_{N\Z^k}f-\res_{\Z^k}f$ is t-summable. Thanks to the homogeneity assumption
and (\ref{gset}), we get
\begin{equation}\label{torsion} 
(N^{k-s}\rho(N)-1)\res_{\Z^k}f\mbox{ is t-summable on }\Q^k.
\end{equation}
It follows that $\res_{\Z^k}f$ is h-summable if $(N^{k-s}-1)\neq 0$, and that
its sum is given by
\begin{equation}\label{useful}
\sum{0\neq n\in\Z^k}f(n)=(N^{k-s}-1)\inverse\Sigma\{\sum{0\neq n\in\Z^k}f(n)z^n:1\neq z\in\T^k,z^N=1\}.
\end{equation}
If $s\neq k$, one may choose such a natural number $N$.  If 
 If $\epsilon=-1$, then $(1+\rho(-1))\res_{\Z^k}f=0$
so it is h-summable in this case as well.  This proves part (1).

Fix $x\in\R^k$. Let $f_x(v)=f(x+v)$ for every $v\in\R^k$. 
Utilizing the Taylor expansion of $f$ at $v\in\R^k$ we obtain
\[f(v+x)=g_0(v)+g_1(v)+...+g_{m-1}(v)+\mathrm{Rem}_m(v)\mbox{ when }\|v\|>\|x\|\]
We note that $g_i(v)=\pd_x^if(v)/i!$ is \df on $\R^k\setminus\{0\}$ and homogeneous of type $(
-(s+i),(-1)^i\epsilon)$. By part (1) we see that $\res_{\Z^k}g_i$ is h-convergent.
  If $m+\Real (s)=h>k$ we see that $\mathrm{Rem}_m(v)$ is $\O(\|v\|^{-h})$
 as $\|v\|\to\infty$, and so  $\mathrm{Rem}_m|_{\Z^k}$ is in $\abss(\Z^k)$. It follows that
 $\res_{\Z^k}f_x$, being a finite sum of h-convergent series, is itself h-convergent.

\end{proof}
\begin{lem}\label{failure} 
If $f:\R^k\setminus\{0\}\to\C$ is $\t{C}^1$, homogeneous
of type $(-k,1)$ , and if $\sum{0\neq n\in\Z^k}f(n)$ is h-summable, then $f=0$.
\end{lem}
\begin{proof} We first assume that $f$ is \df. We set $f(0)=0$. 
Recall the subspaces
 \[\abss_t(\Q^k)\subset
\abss_h(\Q^k)\subset \C^{\Q^k}=V\]
as in Proposition \ref{summary}. 
From (\ref{torsion}), we see that $(\rho(N)-1)\res_{\Z^k}f\in\abss_t(\Q^k)$
for every natural number $N$. In addition we have 
$\rho(-1)\res_{\Z^k}f=\res_{\Z^k}f$.
 It follows that
\\(i) the image of $\res_{\Z^k}f$ in $V/\abss_t(\Q^k)$
 is invariant under the action of $\rho(\Q\units)$. 

From (i) and the definition of h-summability, we get
\\(ii) $\res_{\Z^k}f$ is h-summable if and only if it is t-summable.

By remark \ref{req} with $\G_1=\Z^k$ and $\G_2=\Q^k$, we note:
\\(iii) $\res_{\Z^k}f$ is t-summable if and only if $f|_{\Z^k}:\Z^k\to\C$ is t-summable.

For every $v\in\Z^k$, we observe that $n\mapsto f(n+v)-f(n)$ gives an element
of $\abss(\Z^k)$. It follows that $f|_{\Z^k}$ is t-summable if and only if
it belongs to $\abss(\Z^k)$. Comparing the integral and the sum 
on conical regions, we see that if $f|_{\Z^k}$ is in $\abss(\Z^k)$, then
$f$ is identically zero.

The reader is left to check that the $\t{C}^1$ hypothesis is sufficient for
the validity of the above argument.

\end{proof}
\begin{prop}\label{hpol} If $f$ is a polynomial, the series $\sum{n\in\Z^k}f(n)z^n$ is 
$h$-summable for all $z\in\T^k$,and the sum of this series is zero.

\end{prop}
\begin{proof} By Lemma~\ref{tpol} we see that this series is t-summable and that its sum
is zero for $1\neq z\in\T^k$. For $z=1$ and $f$ homogeneous, the h-summability is
contained in prop.~\ref{hh}, where it is also shown that the sum of this series
is a linear combination of $\sum{n\in\Z^k}f(n)z^n$for certain $z\neq 1$,
 and is thus equal to zero.  Linearity proves the result for all polynomials.

\end{proof}
\begin{prop} \label{prophomholh}

Let $f,g$ be as in corollary \ref{corhomholt}, with $\epsilon=1$. 
Then the h-convergent sum
$h(s)=\sum{0\neq n\in\Z^k}f(n)^{-s}g(n)$ is a holomorphic function of $s\neq (k+s_0)$.
Furthermore, $h$ has a simple pole at worst at $(k+s_0)$.  
\end{prop}
\begin{proof} Both assertions follow from equation (\ref{useful}) for $f^{-s}g$. The t-convergent
sums on the right side of that equation 
are holomorphic by Corollary \ref{corhomholt}, and $(N^{k+s_0-s}-1)$ has a simple
zero at $s=k+s_0$.

\end{proof}

We discuss next an inhomogeneous situation: 
express  $P\in \R[x_1,x_2,...,x_k]$ as the sum of homogeneous polynomials:
 $P=P_0+P_1+...+P_d$. Assume that $P_d(x)>0$ for all nonzero $x\in \R^k$. Evidently,
 $K=\{x\in\R^k:P(x)\leq 0\}$ is compact. Let $F\subset \Z^k$ be any finite subset that contains
 $K\cap \Z^k$.
\begin{thm}\label{special}  
With $P$ and $F$ as above,  let  $G(s)$
denote the sum of the series $\sum{n\in\Z^k\setminus F}P(x)^{-s}$ whenever this series
 is h-convergent.  Let $E=\{(k-j)/d:j\in 2\Z, j\geq 0\}$. Then 
\begin{enumerate}\item
 $G(s)$ is defined for all $s\notin E$, and $G|_{\C\setminus E}$ is holomorphic.
 This function has meromorphic continuation with simple poles, denoted by $\tilde{G}(s)$, to all
 of $\C$.
 \item  Let $s_0\in E$. If $G(s_0)$ is defined, then $\tilde{G}$ is holomorphic at $s_0$. But $\tilde{G}(s)\neq G(s)$ in general.
  \item If $s\in\Z,s\leq 0$, then $G(s)$ is defined and equals  $-\Sigma\{P(n)^{-s}:n\in F\}$.
  Therefore $\tilde{G}$ is holomorphic at all such $s$.
  \item If $k$ is odd, then $\tilde{G}(s)=-\Sigma\{P(n)^{-s}:n\in F\}$ when $s\in\Z,s\leq 0$. 
\end{enumerate}
\end{thm}

\begin{rem} The requirement that $k$ is odd in part (4) can be altered
in the following manner. We consider instead the
h-convergence of $Q.P^{-s}|_{\Z^k\setminus F}$ where $Q$ is a homogeneous polynomial of degree $e$ and once again denote by $\tilde{G}$ the sum of the series extended meromorphically. 
If $s\in\Z$ and $s\leq 0$, then $\tilde{G}(s)$ is once again
the negative of the sum 
$Q(n)P(n)^{-s}$ taken over $n\in F$, \emph{under the assumption that
$(e+k)$ is odd}. The proof of this is a notational modification of the proof
of the theorem given below. 

\end{rem}

\begin{proof} The statement of the theorem remains unaffected if $F$ is replaced by
a larger finite subset. 
So we will assume that $0\in F$.

Write $P=P_d(1-R)$, then $R(x)$ is $\O(\|x\|\inverse)$ as $\|x\|\to\infty$,
so we may approximate $(1-R)^{-s}$ through the power series expansion:
 \begin{equation}\label{formal}
 (1-t)^{-s}=\overset{\infty}{\sum{e=0}}a_e(s)t^e,A_m(s,t)=
 \overset{e=m-1}{\sum{e=0}}a_e(s)t^e,E_m(s,t)=(1-t)^{-s}-A_m(s,t)
 \end{equation}
 noting that $|E_m(s,t)|\leq C(m,r)|t|^m$ if $|s|<r$ and $|t|\leq 1/2$. We deduce
 bounds $|E_m(s,R(x))|\leq C'(m,r)\|x\|^{-m}$ valid when $|s|<r$ and $\|x\|>M(m,r)$.
 Choose $m$ so that $m+d\Real(s)>k$. Let $W=\{s\in\C:|s|<r\}$. For every $s\in W$,
 we see that $x\mapsto
 P_d(x)^{-s}E_m(s,x)$ for $x\in\Z^k\setminus F$ is in $\abss(\Z^k\setminus F)$, and 
 furthermore, these members of $\abss(\Z^k\setminus F)$ lie in a bounded set. By remark
  \ref{cont}, we see that
 
 \begin{equation}\label{holhol}
 s\mapsto e_m(s)=\sum{n\in\Z^k\setminus F}P_d(n)^{-s}E_m(s,R(n))
\mbox{ is holomorphic on }|s|<r.
 \end{equation}
We note that $R$ belongs to the $\Z$-graded ring obtained by adjoining $P_d\inverse$
 to $\R[x_1,x_2,...,x_k]$.  We write 
$R^i=(R^i)_i+(R^i)_{i+1}...+(R^i)_{id}$ where
 each $(R^i)_j$ is homogeneous of degree $(-j)$. 
 Putting $t=R(x)$ in (\ref{formal}) and multiplying by $P_d(x)^{-s}$ we get
 \begin{equation}\label{sum0}
 P(x)^{-s}=P_d(x)^{-s}E_m(s,x)+ \sum{j}P_d(x)^{-s}B_j(s,x)\mbox{ with }B_j(s,x)=\sum{e}a_e(s)(R^e)_j.
 \end{equation}
 By Proposition \ref{hh}, the term 
$P_d(x)^{-s}(R^e)_j$, being homogeneous of degree $(-(ds+j),(-1)^j)$
 , is h-summable when restricted to $\Z^k\setminus F$, unless $ds+j=k$ and $j\in 2\Z$. 
 Thus it is h-summable if $s\notin E$; denote its sum by $g_{e,j}(s)$. By (\ref{sum0}), 
 we see that the series in question has been expressed as a linear combination of
 h-summable series. So we conclude
 \begin{equation}\label{sum1}
 \mbox{If }|s|<r,s\notin E,\mbox{ then }G(s)\mbox{ is defined and equals }e_m(s)+\sum{j}\sum{e}a_e(s)g_{e,j}(s).
 \end{equation} 
Prop. \ref{prophomholh}(1), and (\ref{holhol}) with arbitrarily large $r$,
combine to prove part (1). 

Next, we take $s_0\in E$ and then study (i) the behaviour of $\tilde{G}$
on the region  $U=\{s:|s-s_0|<2/d\}$, and (ii) the h-summability of
$P^{-s_0}|_{\Z^k\setminus F}$. By assumption, we have $p\in 2\Z,p\geq 0$
 and $ds_0+p=k$. 
For every $s\in U$, the restrictions of
all the functions listed in (\ref{sum0}), with the exception of
 $P_d(x)^{-s}B_{p}(s,x)$, when restricted to $\Z^k\setminus\{0\}$ are 
h-summable. Furthermore, their sums are holomorphic functions on $U$.
It only remains to consider $B_{p}(s,x)$, which it is more convenient 
to express as the finite sum:

\[B_{p}(s,x)=V_0(x)+(s-s_0)V_1(x)+...+(s-s_0)^qV_q(x)\]
where all the $V_i(x)\in\R[x_1,...,x_k]_{P_d}$  are homogeneous of degree 
$(-i)$.
Let $L_i(s)=\sum{n\in\Z^k\setminus F}P_d(n)^{-s}V_i(n)$ for all $s\neq s_0$.
By  \ref{prophomholh}, we see that $(s-s_0)^iL_i(s)$ is holomorphic on $U$
for all $i>0$. Putting this together, we see:
\\(A) $\tilde{G}-L_0$ is holomorphic on $U$, and
\\(B) $\sum{n\in\Z^k\setminus F}P(n)^{-s_0}-P_d^{-s_0}(n)V_0(n)$ is h-summable.

Now assume that the given series is h-summable at $s_0$. It follows from
(B) that $P_d^{-s_0}V_0|_{\Z^k\setminus F}$ is also h-summable. By (\ref{failure}), we see that $P_d^{-s_0}V_0=0$. It follows that $P_d^{-s}V_0=0$ and therefore $L_0$ vanishes as well. We conclude from (A) that $\tilde{G}$ is
holomorphic on $U$. We see however that $\tilde{G}(s_0)=G(s_0)+b$ where
$b$ is the residue of $L_1(s)$ at $s_0$. If $k=2, P(x,y)=x^2+y^2-x, s_0=0$,
 one checks that $b$ is precisely the residue of $\zeta_K(s)$ at $s=1$,
where $K=\Q(\sqrt{-1})$, which is $\pi/4$. This completes the proof of part (2).

Proposition \ref{hpol} implies part (3). 

Note that $d\in 2\Z$ and therefore $\Z\cap E$ is empty when $k$ is odd. 
Part (4) is now implied by parts (1) and (3). This completes the proof of the theorem.

\end{proof}


\section{the Formalism of the Poisson Formula}
A nice account of the Poisson
formula is to be found in the books of Lang and Weil, \cite{L} and \cite{BNT}.
The view of this formula 
taken here is more sophisticated; it is borrowed from Weil \cite{acta}, Lion-Vergne \cite{LV} and Mumford \cite{M}.
In essence, it is simply the Poisson formula, for the original function replaced by a translate and then 
multiplied by a character, as displayed in (\ref{GPSF}).  This development  expresses 
the Fourier transform as the composite of linear
 operators that are defined on Frechet spaces, such as the Schwartz space 
$\sx$ of \df functions
of rapid decay on $X$  and $\smdouble$ defined below. These operators extend to certain spaces of distributions. 

The Poisson intertwining operator $T_B:\sx\to\smdouble$, given in (\ref{PIO}), plays a significant role in this section.   Formula  (\ref{tdfn}) defines $T_Bf$ as a sum
taken over $\g\in\G$ for every $f\in\sx$. 
The extended version of $T_B$ however, now defined for all tempered distributions, is handed to us as a package, not as a sum indexed by
$\g\in\G$.
  We think of Theorem \ref{breakup}
as the `true' Poisson formula for tempered distributions.  For a tempered distribution $u$,  the distribution $T_Bu$ is shown to be the sum of the distributions 
$u_{\g}$ defined in (\ref{ugdfn}). 

This permits us to sieve out  the obvious contributions to the singularities of $T_Bu$  at the origin to obtain
$T_B^{reg}u$ in Definition \ref{regdfn}. The next objective is to give sufficient conditions
that ensure $T_B^{reg}u$ is \df in a neighborhood of zero (see Proposition \ref{triple}).

Except for the factor $2\pi$ in the Fourier transform, the notation here
for distributions, and operations on distributions, is completely 
consistent with that of Hormander's book \cite{H}.
The facts on distributions that we use are seen in a first course
on the subject: they are contained in   
vol.1, chapter 7 , \cite{H}, and also to be found in chapters 6,7 of \cite{R1}.
\subsection{Review of the Poisson Formula}
Our data consists of a four-tuple $(X,X',B,\Gamma)$ where
\begin{enumerate}

\item[(a)] $X$ and $X'$ are finite dimensional $\R$-vector spaces, 
\item[(b)] $B:X\times X'\to\R$ is a non-degenerate bilinear form, and
\item[(c)]  $\Gamma\subset X$ is a lattice, i.e. $\Gamma$ is discrete and $X/\Gamma$
is compact.
\end{enumerate}
Every  $(X,X',B,\Gamma)$ as above produces its \emph{dual} $(X',X,B',\Gamma')$ given by
\begin{equation}\label{dual}
B'(x',x)=-B(x,x')\forall x\in X,x'\in X'
\end{equation}
with dual lattice
\begin{equation}\label{dual'}
\Gamma'=\{\gamma'\in\Gamma':B(\gamma,\g')\in\Z\forall \gamma\in\Gamma\} 
\end{equation}
The compact torus $Z$ is defined by
\begin{equation}\label{torus}
Z=X\times X'/\G\times\G'
\end{equation}
\\The Haar measure on $X$ is chosen so that $\t{vol}(X/\G)=1$. The integral
of a function with respect to this Haar measure will be denoted by 
$\int_Xf(x)dx$ or even simply by $\int_Xf$. We put 
\begin{equation}\psi(t)=\exp(2\pi\sqrt{-1}t)\,\forall t\in\R
\end{equation}  and
\begin{equation}\label{exp}
\psi_B(x,x')=\psi(B(x,x'))\text{ for all }x\in X,x'\in X'.
\end{equation}
\\ We recall that the Schwartz space of $X$, denoted by 
$\sx$, is the collection of \df $\C$-valued functions $f$ defined on $X$
for which $\semi{f}{M}{N}<\infty$ for all non-negative integers $M$ and $N$,
 where 
\begin{equation}\label{seminorm}
\semi{f}{M}{N}=\sup\{(1+\|x\|)^N\|v\|^{-m}|\pd_v^mf(x)|:x\in X,0\neq v\in X, 0\leq m\leq M\}
\end{equation}
In the above, $\pd_v$ denotes the directional derivative. Norms on both
$X$ and $X'$ are chosen arbitrarily and fixed once and for all. 
The above 
semi-norms $\semi{\,}{M}{N}$ give $\sx$ the structure of a topological
 vector space. 
\\For all $f\in\sx$, its Fourier transform $\fb f$ is the function on $X'$
defined by the absolutely convergent integral
\begin{equation}\label{fdfn}
\fb(f)(x')=\int_Xf(x)\psi_B(x,x')dx.
\end{equation}
If $f\in\sx$, then $\fb f\in\sxp$, and in fact
$\fb:\sx\to\sxp$ is continuous. This statement 
follows from the standard identies below, valid for all $u\in\sx$
\begin{equation}\label{weyl}
\pd_{x'}\fb u=2\pi\sqrt{-1}\fb B_{x'}u
\end{equation}
and 
\begin{equation}\label{weyl2}\fb\pd_xu=2\pi\sqrt{-1}B'_x\fb u
\end{equation} 
where $B_{x'}(x)=B(x,x')$ and $B'_{x}(x')=B'(x',x)$ for all $x\in X,x'\in X'$.

Let us return to (\ref{fdfn}), the Fourier integral. This integral
may be computed, first by summing over $x+\G$, and
then integrating the resulting function on $X/\G$. In the summation over
$x+\G$, there is a common factor $\psi_B(x,x')$ which we suppress for
the moment. Thus,  for $f\in\sx$,
we define $T_Bf:X\times X'\to \C$ by the formula
\begin{equation}\label{tdfn}
T_Bf(x,x')=\underset{\g\in\G}{\Sigma}f(x+\g)\psi_B(\g,x'),
\end{equation}

\emph{For the rest of this section and in (\ref{tinvariance}) in particular,  we employ the following notation}:
let $\s{D}(X\times X')$ denote the space of distributions on $X\times X'$. The map that turns 
$\s{D}(X\times X')$ into a module over the ring of infinitely differentiable functions on $X\times X'$  will be denoted by $(h,u)\mapsto h.u$  for all infinitely differentiable  
$h:X\times X'\to\C$
and all $u\in\s{D}(X\times X')$. 

We note that the property
\begin{equation}\label{tinvariance}
u\mbox{ and }\psi_B.u\mbox{ are translation invariant by }
\G'\mbox{ and }\G\mbox{ respectively,}
\end{equation}
is valid for $u=T_Bf$ for all $f\in\sx$. 

We define the space

\begin{equation}\label{smdouble}
\smdouble
\end{equation}
to be collection of 
of infinitely differentiable functions  $u:X\times X'\to\C$ 
that satisfy (\ref{tinvariance}).  Formula (\ref{tdfn}) then defines the \emph{Poisson intertwining operator} $T_B$
\begin{equation}\label{PIO}
T_B:\sx\to\smdouble
\end{equation}
For $f,g\in\smdouble$, the function $f.\o{g}:X\times X'\to\C$ descends to a function
$(f.\o{g})_d$ on the torus $Z$ defined in (\ref{torus}). We define
\begin{equation}\label{innsm}
\dot{f}{g}=\int_Z(f.\o{g})_d
\end{equation}
$\smdouble$ has the structure of a Frechet space given by the seminorms $P(v,v',m,n)$ defined for $v\in X, v'\in X'$ and non-negative integers $m,n$ by
\[P(v,v',m,n)u=\sup\{|\pd_v^m\pd_{v'}^nu(x,x'))|: (x,x')\in K\}
\]
where $K\subset X\times X'$ is a compact subset with the property that the composite \linebreak $K\hkr X\times X'\to Z$ is onto. The resulting topology on
$\smdouble$ is seen to be independent of the choice of $K$.

 The theory of Fourier series for $\t{C}^{\infty}$ functions on 
$X'/\G'$ suffices to deduce  the next proposition. 
\begin{prop}\label{Fseries}\noindent
\begin{enumerate}
\item $T_B:\sx\to\smdouble$ is an isomorphism of topological vector spaces,
\item $\dot{f}{g}=\dot{T_Bf}{T_Bg}$ for all $f,g\in\sx$,
\item The operator $S_B:\smdouble\to\sx$ given by $S_Bh(x)=\int_{X'/\G'}h(x,x')dx'$ 
is the inverse of $T_B$. 
\end{enumerate}

\end{prop}
The passage preceding (\ref{tdfn}) shows that  the Fourier transform of $f\in\sx$ is then given by the formula:
\begin{equation}\label{fourierdec}
\fb f(x')=\int_{X/\G}T_Bf(x,x')\psi_B(x,x')dx.
\end{equation}
All the above considerations are valid also for the dual $(X',X,B',\Gamma')$. The space defined in (\ref{smdouble})  for  $(X',X,B',\Gamma')$
will be denoted by $\smdoublep$. The operators given in (\ref{PIO}) and part (3) of Proposition \ref{Fseries}  for the dual $(X',X,B',\Gamma')$ will be denoted by
$T_{B'}$ and $S_{B'}$ respectively. The Fourier transform for the dual $(X',X,B',\Gamma')$ will be denoted by $\s{F}_{B'}$.
\\We have the isomorphism 
\begin{equation}\label{switchiso}
\sigma_B:\smdouble\to\smdoublep
\end{equation}
given by
\begin{equation}\label{switchdfn}
(\sigma_Bh)(x',x)=\psi_B(x,x')h(x,x')\,\,\forall h\in\smdouble
\end{equation}

Proposition~\ref{Fseries}, (\ref{fourierdec}) and the   negative sign in the definition of $B'$
combine to prove the following proposition. 
\begin{prop}\label{FPSF} \noindent
 \begin{enumerate}
 \item The operators $\sigma_B,T_B,T_{B'}$ are isomorphisms
 of topological vector spaces, and their inverses are $\sigma_{B'},S_B,S_{B'}$
 respectively. 
\item  $\fb=T_{B'}\inverse\circ\sigma_B\circ T_B$ and 
$\s{F}_{B'}=T_B\inverse\circ\sigma_{B'}\circ T_{B'}$, and thus $\fb$ and $\s{F}_{B'}$ are inverses of each other.
\item  The  property ``$\dot{\Phi f}{\Phi g}=\dot{f}{g}$ for all $f,g$ in the domain of $\Phi$'' holds for $\Phi=\sigma_B, T_B, T_{B'}$.
\end{enumerate}
\end{prop} 
Part (2) of Proposition~\ref{FPSF} shows $T_Bf=\sigma_{B'}T_{B'}\s{F}_{B'}f$
 and this reads as:

\begin{equation}\label{GPSF}
\underset{\g\in\G}{\Sigma}f(x+\g)\psi_B(\g,x')=\underset{\g'\in\G'}{\Sigma}
\fb(f)(x'+\g')\psi_{B'}(x'+\g',x)\,\,\forall f\in\sx
\end{equation}
The standard form of the Poisson formula, given below, is obtained by putting $(x,x')=(0,0)$.
\begin{equation}\label{PF}
\underset{\g\in\G}{\Sigma}f(\g)=\underset{\g'\in\G'}{\Sigma}
\fb(f)(\g')\,\,\forall f\in\sx
\end{equation}
\begin{rem}\label{extension} The space of distributions (resp. tempered distributions) on a finite dimensional real vector space $V$ is denoted by $\s{D}(V)$
 (resp. $\s{S}(V)^*$).
\\(A) We shall define $\dot{}{}:\sx^*\times\sx\to\C$ by
\[\dot{u}{f}=u(\o{f})\mbox{ for all }u\in\sx^*,f\in\sx.\]
(i) The restriction of $\dot{}{}$ to $\sx\times\sx$ is the standard inner product.
\\(ii) Every continuous linear functional on $\sx$ is given by $f\mapsto\o{\dot{u}{f}}$ for a unique
$u\in\sx^*$.
\\ \\(B) Let $\disdouble$ be the space of distributions $u$ on $X\times X'$
 that satisfy (\ref{tinvariance}).
We will define 
\[\dot{}{}:\disdouble\times\smdouble\to\C\]
as follows. For $u\in\disdouble,h\in\smdouble$, the distribution $\o{h}.u$ is invariant
under translation by $\G\times\G'$ and therefore descends to a distribution  $(\o{h}.u)_d$
on the torus $Z$ defined in (\ref{torus}). Denoting by $1_Z$ the constant function $1$ on $Z$,
 we define $\dot{u}{f}=(\o{h}.u)_d1_Z$. 
 \\(i) If $u$ is also in $\smdouble$, then this definition
 of $\dot{u}{h}$ agrees with the formula of (\ref{innsm}). 
 \\(ii) Every continuous linear functional on $\smdouble$ is given by $h\mapsto\o{\dot{u}{h}}$
  for a unique $u\in\disdouble$. One sees this by identifying $\smdouble$ with the global
  $\t{C}^{\infty}$ sections of a unitary line bundle $L$ on $Z$. The compactness of $Z$ 
  then gives the identification of $\o{\smdouble^*}$ with the space of global sections of $L\otimes D$
  where $D$ is the sheaf of distributions on $Z$. The latter space is canonically identified
  with $\disdouble$.
  
  We also have (A) and (B) above for the dual $(X',X,B',\G')$, namely
\\(A$'$) $\dot{}{}:\sxp^*\times\sxp\to\C$ and
\\(B$'$):$\dot{}{}:\disdoublep\times\smdoublep\to\C$.
\begin{prop}\label{FPSF2} For $F=\sx,\sxp,\smdouble,\smdoublep$, 
  \\let $F_e=\sx^*,\sxp^*,\disdouble,\disdoublep$ respectively.
  
  Every isomorphism $U:F_1\to F_2$ that occurs in Proposition~\ref{FPSF} extends uniquely to an isomorphism $U_e:(F_1)_e\to (F_2)_e$ that is specified uniquely by
  \begin{equation}\label{extnformula}
\dot{U_eu}{Uh}=\dot{u}{h}\mbox{ for all }u\in (F_1)_e,h\in F_1
\end{equation}
To simplify notation, the extended operator $U_e$ will be denoted once  again by $U$.

The identities
\[\fb=T_{B'}\inverse\circ\sigma_B\circ T_B\text{  and  }
\s{F}_{B'}=T_B\inverse\circ\sigma_{B'}\circ T_{B'}\] are valid for the extended operators as well.
\end{prop} 
  \begin{proof}
 Properties (i) and (ii) of $\dot{}{}:F_e\times F\to\C$ listed in Remark \ref{extension} (A) and (B), combined with the last
 assertion of Proposition~\ref{FPSF}, prove this statement.
  \end{proof}
\end{rem}
\subsection{Expression of  $T_Bu$ as a sum over $\g\in\G$ }
By Proposition \ref{FPSF2}, $T_Bu$ has been defined
for all $u\in\sx^*$. Our next goal, Theorem \ref{breakup} below, is to 
prove the validity of (\ref{tdfn}) for $u\in\sx^*$. For this purpose,
we first give meaning to the $\g$-th term in (\ref{tdfn}) for every 
distribution $u$ on $X$ in the standard manner.

In order to obtain the formula
\begin{equation}\label{Imot}
\int_{X\times X'}f(x+\g)\psi_B(\g,x')\phi(x,x')dxdx'=\int_Xf(x)I_{\g}\phi(x)dx
\end{equation}
for all $\g\in\G,f\in\sx$ and all test functions $\phi\in\test(X\times X')$, 
we define $I_{\g}\phi\in\test(X)$  by
  \begin{equation}\label{Idfn}
I_{\g}\phi(x)=\int_{X'}\psi_B(\g,x')\phi(x-\g,x')dx'.
 \end{equation}

 For a distribution $u$ on $X$ we define the distribution $u_{\g}$ on $X\times X'$ by
\begin{equation}\label{ugdfn} u_{\g}\phi=u(I_{\g}\phi)\,\forall \mbox{test functions }
\phi\in \test(X\times X').
\end{equation}
By (\ref{Imot}), we see that if $u=f\in\sx$, then $u_{\g}$ is given by the function
$(x,x')\mapsto f(x+\g)\psi_B(\g,x')$, as desired. Lemma \ref{convergence} is required 
to show that $\underset{\g\in\G}{\Sigma}u_{\g}$ converges to a distribution on $X\times X'$.

For $\phi\in\test(X\times X')$ we define
\[\semitest{\phi}{M}{N}=\sup\{\|x\|^{-m}\|x'\|^{-n}
\|\pd_x^m\pd_{x'}^n\phi\|_{\infty}:0\neq x\in X,0\neq x'\in X', 0\leq m\leq M,0\leq n\leq N\}\]
\begin{lem}\label{convergence} Let $K\subset X$ and $K'\subset X'$  
be compact subsets. Then, for every
$M,N\geq 0$, there is a constant $C(K,K',M,N)$ with the property that the inequality
\begin{equation}\label{conv0}
\underset{\g\in\G}{\Sigma}\semi{I_{\g}\phi}{M}{N}\leq C(K,K',M,N)\semitest{\phi}
{M}{N+a}
\end{equation}
holds for every $\phi\in\test(K\times K')$, with $a=1+\dim X$ and 
notation as in (\ref{seminorm}).
\end{lem}
\begin{proof} Putting $L_x(x')=\phi(x,x')$, we first note that
$I_{\g}\phi(x)=\s{F}_{B'}L_{x-\g}(-\g)$ for all $x\in X,\g\in\G$.

$\int_{X'}(1+\|x'\|)^{-a}<\infty$ and the
standard identities (\ref{weyl}) imply 
\begin{equation}\label{conv1}
\semi{\s{F}_{B'}f'}{0}{N}\leq C_1(N)\semi{f'}{N}{a}\forall f'\in\sxp
\end{equation}
for every $N$. Let $b=1+\sup\{\|x'\|:x'\in K'\}$. We may now
rewrite the above inequality for $f'=L_x$ in the form below:
\begin{equation}\label{conv2}
(1+\|y\|)^N|\s{F}_{B'}L_x(y)|\leq b^aC_1(N)\semitest{\phi}{0}{N}\mbox{ for all }
x,y\in X
\end{equation}
From the compactness of $K$, we get $c>1$ for which
\begin{equation}\label{conv3}
c\inverse (1+\|y-x\|)\leq (1+\|y\|)\leq c(1+\|y-x\|)\mbox{ for all }x\in K,
y\in X.
\end{equation}
Replacing $N$ by $N+a$ in (\ref{conv2}), we may now rewrite that inequality
in the form
\begin{equation}\label{conv4}
(1+\|y\|)^a(1+\|y-x\|)^N|\s{F}_{B'}L_x(y)|\leq b^ac^aC_1(N+a)\semitest{\phi}{0}{N+a}\mbox{ for all }
x,y\in X.
\end{equation}
Replacing $(x,y)$ in the above by $(x-\g,-\g)$ we get
\begin{equation}\label{conv5}
(1+\|\g\|)^a(1+\|x\|)^N|I_{\g}\phi(x)|\leq b^ac^aC_1(N+a)\semitest{\phi}{0}{N+a}\mbox{ for all }
x\in X,\g\in\G.
\end{equation}
Taking $C(K,K',0,N)=b^ac^aC_1(N+a)\underset{\g\in\G}{\Sigma}(1+\|\g\|)^{-a}$,
 we have proved (\ref{conv0}) when $M=0$.

We note that $I_{\g}\pd_x\phi=\pd_xI_{\g}\phi$ for all $x\in X$.
 We then see that
 the inequality (\ref{conv0}) for $(0,N)$ applied to all the partial
derivatives of $\phi$ of order at most $M$ proves (\ref{conv0}) for
$(M,N)$ as well.  

\end{proof}

\begin{lem}\label{WW} Let $\phi\in\test(X\times X')$.  Then $W\phi$ given by
\[W\phi(x,x')=\underset{\g\in\G}{\Sigma}\underset{'\g\in\G'}{\Sigma}\phi(x+\g,x'+\g')\psi_B(\g,x')\]
has the properties below
\begin{enumerate}
\item $W\phi\in\smdouble$
\item  $u(\o{\phi})=\dot{u}{W\phi}$ for all $u\in\disdouble$ with $\dot{}{}$ as given in 
remark~\ref{extension}(B)
\item $W:\test(X\times X')\to \smdouble$ is surjective.
\end{enumerate}
\end{lem}
\begin{proof}Let $\Omega\subset X\times X'$ be an open subset whose translates by
all $(\g,\g')\in\G\times\G'$ are all disjoint. With $Z$ as in (\ref{torus}), 
let $P:X\times X'\to Z$ denote the projection.  Let $\phi$ be a test function with $\supp(\phi)\subset\Omega$. We observe
\begin{equation}\label{WW1}
h\in\smdouble, \,\,h|\Omega=\phi|\Omega,\,\,\supp(h)\subset P\inverse P\Omega
\iff h=W\phi. 
\end{equation}
Let $u\in\disdouble$. Recall that $\o{W\phi}.u$ descends 
to a distribution $(\o{W\phi}.u)_d$ on $Z$. We have  $\dot{u}{W\phi}=(\o{W\phi}.u)_d1_Z$.  Note that  $(\o{W\phi}.u)_d1_Z=(\o{W\phi}.u)_dg$
if $g:Z\to\C$ is a \df function such that $g(z)=1$ on an open subset of 
$Z$ that contains $P\t{supp}(\phi)$. 
To obtain such a $g$, we choose
another test function $\phi'$ on $X\times X'$  satisfying
\[\supp(\phi')\subset\Omega\mbox{ and }\supp(\phi)\subset\subset (\phi')\inverse\{1\}\]
and let $g:Z\to\C$ be the unique function with $\supp(g)\subset P(\Omega)$ and
$\phi'(z)=g(P(z))$ for all $z\in\Omega$. We then have
\[\dot{u}{W\phi}= (\o{W\phi}.u)_dg= (\o{\phi}.u)(\phi')=u(\o{\phi}.\phi')=u(\o{\phi})\]  
This proves the first two assertions for such $\phi$. The linear span of 
$\test(\Omega)$, taken over all $\Omega$ as above, is the collection of all test functions. Thus the first two assertions of the lemma follow from linearity.
The same reasoning, combined with (\ref{WW1}), proves the third assertion
 as well.
\end{proof}
\begin{lem}\label{JW}  Let $J_{\g}\o{\phi}=\o{I_{\g}\phi}$ for all
$\phi\in\test(X\times X'),\g\in\G$. For every $x\in X$, the finite sum 
$\underset{\g\in\G}{\Sigma}J_{\g}(x)$ equals $\int_{X'/\G'}W(x,x')$.
\end{lem}
\begin{proof} Now $J_{\g}\phi(x)=\int_{X'}\psi_B(-\g,x')\phi(x-\g,x')dx'$ . Once again, this
integral may be computed by first summing over $x'+\G'$ and then integrating the resulting
 function on $X'/\G'$. So we get:
\begin{equation}\label{J2}
A_{\g}\phi(x,x')=\underset{\g'\in\G'}{\Sigma}\psi_B(-\g,x')\phi(x-\g,x'+\g')
\mbox{ and }J_{\g}\phi(x)=\int_{X'/\G'}A_{\g}\phi(x,x')
\end{equation}
and now summing the above over $\g\in\G$, we get
\begin{equation}\label{J3}
W\phi(x,x')=\underset{\g\in\G}{\Sigma}A_{\g}\phi(x,x')
\mbox{ and }\underset{\g\in\G}{\Sigma}J_{\g}\phi(x)=
\int_{X'/\G'}W\phi(x,x').
\end{equation}
\end{proof} 
For a tempered distribution $u$ on $X$, part (4) of Theorem \ref{breakup} expresses $T_Bu$ as a sum of distributions on $X\times X'$. Part (5) of the theorem 
is the identity $T_Bu=\sigma_{B'}T_{B'}\fb u$, which is the generalized Poisson formula  (\ref{GPSF}) when $u=f\in\sx$.

\begin{thm}\label{breakup} \textbf{(Poisson Formula for tempered distributions)}
\noindent\\
Let $u$ be a tempered distribution on $X$.
\begin{enumerate}
\item Let $S$ be a subset of  $\G$. For every $\phi\in\test(X\times X')$, the
series $\underset{\g\in S}\Sigma u_{\g}\phi$ converges absolutely (see
 (\ref{ugdfn}) for the definition of $u_{\g}$). 
\item Denote the sum of the above series by $u_S\phi$. 
Then $\phi\mapsto u_S\phi$ defines a distribution on $X\times X'$. 
\item If $S$ is the disjoint union of subsets $S',S''\subset\G$, then 
$u_S=u_{S'}+u_{S''}$.
\item The distribution $T_Bu$ given by Proposition \ref{FPSF2} equals 
the above $u_S$ 
when $S=\G$.
\item $u_{\G}=\sigma_{B'}(\fb u)_{\G'}$.
\end{enumerate}

\end{thm}
\begin{proof}Let $u\in\sx^*$. By definition, there are some $M,N,C$
 so that $|u(f)|\leq C\semi{f}{M}{N}$. Recall that $u_{\g}\phi=u(I_{\g}\phi)$
for all $g\in\G$.  Lemma \ref{convergence} now shows that 
\[\sum{\g\in S}|u(I_{\g}\phi)|\leq C.C(K,K',M,N)\semitest{\phi}{M}{N+a}\forall
\phi\in\test(K\times K').\]
That proves part (1). The same upper bound holds
for $|u_S\phi|$ as well, and now part (2) follows from the definition of a distribution.
Part (3) is evident. 

We now address part (4). It is clear that $u_{\G}$ is in $\disdouble$.
By Proposition \ref{FPSF2}, the identity $T_Bu=u_{\G}$ is equivalent
to the equality $\dot{u_{\G}}{T_Bf}=\dot{u}{f}$ for all $f\in\sx$. 
Now $S_B$ in Proposition \ref{FPSF} is surjective, and so is $W$ (see
 lemma \ref{WW}(3). So it suffices to check this equality for $f=S_BW\phi$ for all $\phi\in\test(X\times X')$. 
In other words, we have to check
\[
\dot{u_{\G}}{W\phi}=\dot{u}{S_BW\phi}\mbox{ for all }\phi\in\test(X\times X')
\] 
We note
\[\dot{u_{\G}}{W\phi}=u_{\G}\o{\phi}=u(\underset{\g\in\G}{\Sigma}I_{\g}\o{\phi})=u(\underset{\g\in\G}{\Sigma}\o{J_{\g}\phi)}=u\int_{X'/\G'}\o{W\phi}=
u(\o{S_BW\phi})=\dot{u}{S_BW\phi}
\]
from \ref{WW}(2) and \ref{JW}, and this completes the proof of part (4). 

The validity of part (4) for both $(X,X',B,\G)$ and  $(X',X,B',\G')$, combined with Proposition \ref{FPSF2} 
 now implies part (5), and therefore completes the proof of the theorem.

\end{proof}
\subsection{Removal of singularities and the definition of $T^{reg}_Bu$}

\begin{lem}\label{consistency}Let $u\in\sx^*$. Let $U\subset X$ be an open subset so that the restriction $u|U$ is given by a continuous function $f:U\to\C$.
Let $A$ be a subset of $\G$. Let 
$\Omega\subset X$ be an open subset so that $\Omega+A\subset U$.  Assume further that
$\underset{\g\in A}{\Sigma}\|f|_{\g+\Omega}\|_{\infty}<\infty$. 
Then the distribution $u_A$ given in Theorem \ref{breakup}
is given by the continuous function
$x\mapsto \underset{\g\in A}{\Sigma}f(x+\g)\psi_B(\g,x')$  on $\Omega\times X'$.

\end{lem}
\begin{proof}
Let $\phi\in\test(\Omega\times X')$ and let $\g\in A$. We observe that 
$I_{\g}\phi\in\test(U)$ and also that  equation (\ref{Imot}) holds
 (in fact for any continuous function $f:U\to\C$).
It then follows from the definition of $u_{\g}$ given in
(\ref{ugdfn}) that 
\[u_{\g}\phi=
\int_{X\times X'}f(x+\g)\psi_B(\g,x'))\phi(x,x')dxdx'.
\]
Summing over $\g\in S$ we obtain
\[
T_B^Su(\phi)=
\int_{X\times X'}(\underset{\g\in S}{\Sigma}f(x+\g)\psi_B(\g,x'))\phi(x,x')dxdx'
\]
for all test functions $\phi$ with support contained in $\Omega\times X'$.
\end{proof} 
In particular, taking $A=\{\g\}$, we see that if $\g$ is not in the singular
support of $u$, then $u_{\g}|U\times X'$ is \df for a suitable neighbourhood
$U$ of zero in $X$. This leads to the definition below.
\begin{dfn}\label{regdfn}  The regularized distribution $T^{reg}_Bu$  is defined for 
$u\in\sx^*$ under the assumption that the singular supports of $u$ and 
$\fb u$, denoted by $K$ and $K'$ respectively, are both compact. Let
\[
S=(\G\cap K)\cup\{0\}\mbox{ and }S'=(\G'\cap K')\cup\{0\}.
\] 

We define

\[T^{reg}_Bu:=T_Bu-\underset{\g\in S}{\Sigma}u_{\g}-\sigma_{B'}\underset{\g'\in
S'}{\Sigma}
(\fb u)_{\g'}
\]
For the above formula, one should note that $u'_{\g'}$ and $u'_{S'}$ are defined for
arbitrary $u'\in \sxp^*,\g'\in\G',S'\subset \G'$ by applying (\ref{ugdfn}) and \ref{breakup}(2)
to the dual $(X',X,B',\G')$.
\end{dfn}
There is no real reason here to throw $0$ into the definition of $S$
 and $S'$. This has been done for the sole purpose of obtaining
consistency with the usage of $T^{reg}u$ in the next section. 
\begin{prop}\label{triple}  Let $u=v+f+\s{F}_{B'}v'$ where $f\in\sx$ and
$v$ and $v'$ are compactly supported distributions on $X$ and $X'$
 respectively. Then 
\begin{enumerate}
\item
$T_B^{reg}u$ is defined, and is 
\df on $U\times U'$ where $U$ and $U'$ are suitable neighbourhoods of $0$
 in $X$ and $X'$ respectively. 
\item  Let $K,K',S,S'$ be as in definition \ref{regdfn}.
Then part (1) holds for  $U$ and $U'$ with
\[X\setminus U=\{x-\g:x\in K,\g\in\G\setminus S\}\mbox{ and }
X'\setminus U'=\{x'-\g':x'\in K',g'\in\G'\setminus S'\}.\]
\end{enumerate}
\end{prop}

\begin{proof} 
Thanks to Paley-Weiner-Schwartz (see pp 180-185, Chapter 8\cite{R}), we see that $\fb v$
 and $\s{F}_{B'}v'$ are both \df. It follows that the singular supports of
 $u$ and $\fb u$ are precisely the singular supports of $v$ and $v'$ respectively. These sets are compact, and therefore $T^{reg}_Bu$ is defined. We
retain the notation $K,K',S,S'$ introduced in \ref{regdfn}.

For $w\in\sx^*$, let 
\[-Pw=\underset{\g\in S}{\Sigma}w_{\g}\mbox{ and } -Qw=
\sigma_{B'}\underset{\g'\in S'}{\Sigma}(\fb w)_{\g'}.\]

Now $T_B^{reg}u$ is the sum of nine terms, obtained by
applying the three operators $T_B,P,Q$ to the three distributions
 $v,f,\s{F}_{B'}v'$.

Five of these, namely $T_Bf,Pf,Qf,Qv$ and $P\s{F}_{B'}v'$, 
are evidently \df on all of $X\times X'$.  We claim that
$(T_B+P)v, (T_B+Q)\s{F}_{B'}v'$ are \df on
             $U\times X'$ and $X\times U'$ 
respectively, where $U\subset X$ and $U'\subset X'$
are suitable neighborhoods of zero.  This claim implies  part (1) of the proposition.

To check this claim, let $L$ be the support of $v$, and define
the subsets $C,D\subset\G$ by 
\[\{0\}\cup (L\cap \G)=S\sqcup C\mbox{ and } \G=S\sqcup C\sqcup D.\]
 By \ref{breakup}, we see that $T_Bv=v_S+v_C+v_D$.
If $U$ is small enough, we see that $v_{\g}|U\times X'$ is 
\\(i) \df if $\g\in C$, and 
\\(ii) is zero if $\g\in D$.

From the definition of $v_D$ in \ref{breakup} it follows that 
$v_D|U\times X'$ is zero.
 The finiteness of $C$ shows that $v_C$ is \df on $U\times X'$. 
Now $Pv=-v_S$, so we see that $(T_B+P)v$ equals $v_C$ which
is \df on $U\times X'$ as claimed. 

To go further,  we note that (i) and (ii) above 
are valid when  $U$ is the complement of the union of $R_1=\{x-\g:x\in K,\g\in C\}$
 and $R_2=\{x-\g:x\in L,\g\in D\}$.  We may express $v+f$ as
$v_1+f_1$ where $v_1=\phi.v$ and $f_1=f+(1-\phi).v$ for a test function $\phi$
that is 1 on an open subset $V$ that contains $K$. Now the support
of $v_1$ is contained in $V$. It follows that
we may replace $L$ by $K$ in the definition of $R_2$. Because $S\sqcup C\sqcup
D=\G$, we deduce that
$U$ can be chosen to be the complement of
$\{x-\g:x\in K,\g\in\G,\g\notin S\}$.

With  $ r=T_{B'}v'-\sum{\g'\in S'}v'_{\g'}$ 
the same argument for the dual $(X',X,B',\G')$ shows that $r$ is \df on
$U'\times X$ with $U'$ as in part (2) of the proposition.  Because $(T_B+Q)\s{F}_{B'}v'$ 
equals $\sigma_{B'}r$, we see that the remaining half of the claim has also been proved.
This completes the proof of the proposition.

\end{proof}

\section{The Poisson formula for mild singularities}
We apply the considerations of the previous section to the spaces of distributions
given below. The notation and operators introduced there will be employed here as well.

The main result of this section is Theorem \ref{mainthm} formulated for distributions that belong to the space $\HH(X)$. 
The spaces of distributions relevant for this theorem are given in \S6.1. Proposition \ref{Htriple}, part 3, says that 
a distribution $u$
 on $X$ belongs to $\HH(X)$ if and only if $u$ and its Fourier transform are well-behaved at infinity in a certain technical sense (see (\ref{Hdfn1}) below). Part 2 of the same proposition expresses $u\in\HH(X)$ as a sum of three distributions. One checks  the validity of Theorem \ref{mainthm} for each of them separately. Theorem \ref{mainthmintro} is then deduced from Theorem \ref{mainthm}. 
 
 Subsection 6.4 has more definitions of translation invariant integrals and sums. Proposition \ref{strong} relates these definitions with the t-summability of \S 2.1.

\subsection{Spaces of Distributions}
\begin{dfn}\label{Hdfn} We will define six spaces of distributions
on $X$. \\Let $u$ be a distribution on $X$. 
\\(I)  $u$ belongs to $\HH_{\infty}(X)$  if $u$ is \df on the region $\|x\|>R$ for some $R>0$, and
there is some real number $p$ with the property
\begin{equation}\label{Hdfn1} 
\|x\|^m\pd_v^mu(x)| \mbox{ is } \O(\|x\|^{p})\mbox{ as }\|x\|\to\infty\mbox{  for all }v\in X, m\geq 0
\end{equation}
\\(II) $u$ belongs to $\HH_0(X)$ if $u$ is \df on the region $0<\|x\|<R$ for some $R>0$, and
there is some real number $p$ with the property that  
\begin{equation}\label{Hdfn2} 
\|x\|^m\pd_v^mu(x)| \mbox{ is } \O(\|x\|^{p})\mbox{ as }\|x\|\to 0\mbox{  for all }v\in X, m\geq 0
\end{equation}
\\(III) $u$ is in $\HH(X)$ if 
\\(a) $u$ is \df on the region $0\neq x\in X$ and 
\\(b)  $u$ belongs to both $\HH_0(X)$ and  $\HH_{\infty}(X)$. Note, however,
that the $p$'s that appear for $\|x\|\to\infty$ and $\|x\|\to 0$ may
be different from each other.
\\(IV)  $\HH^+_{\infty}(X)=\t{C}^{\infty}(X)\cap \HH_{\infty}(X)$. 
\\(V) $u$ belongs to $\HH_0^+(X)$ if it satisfies
the three conditions: 
\\(i) $u$ is compactly supported, 
\\(ii) the singular support of $u$ is contained in $\{0\}$
\\(iii) $u\in \HH_0(X)$.
\\(VI) $\s{D}_c(X)$ is the space of compactly supported distributions on $X$.
 \end{dfn}
Note that if $u$ belongs to $\HH_{\infty}(X)$, then $u$ is certainly a tempered
distribution, and therefore $\fb u$ is defined.
\begin{prop}\label{Hprop}
$u\in\HH^+_{\infty}(X)\iff\fb u\in\sxp+\HH^+_0(X')$.

\end{prop}
\begin{proof} 
Given $x'\in X'$, the function 
 $x\mapsto B(x,x')$, denoted by $B_{x'}$ in equation (\ref{weyl}),section 5,  will now be
denoted simply by $x'$. 

Let $u=f\in\HH_{\infty}^+(X)=\t{C}^{\infty}(X)\cap\HH_{\infty}(X)$, and let $p\in\R$ be
as in (\ref{Hdfn1}). Choose an integer $k\geq 0$ so that $k-p>\dim X$.
The inequality of (\ref{Hdfn1}) assumed for $\|x\|\mapsto\infty$  now implies

\begin{equation}\label{Hprop1}
(v')^m\pd_v^{m+r+k}f\in\abss(X)\mbox{ for all }m\geq 0,r\geq 0, v'\in X',v\in X.
\end{equation}
From (\ref{weyl}), we deduce
\begin{equation}\label{Hprop2}
\pd_{v'}^mv^{m+r+k}\fb f\in\t{C}_0(X')\mbox{ for all }m\geq 0,r\geq 0, v'\in X',v\in X.
\end{equation}
From Weyl's
commutation relations, one checks that for every $h\in\Z$, 
\begin{equation}\label{weyl3}
\{\pd_{v'}^mv^n:n-m=h,v\in X,v'\in X'\}\mbox{ and }\{v^n\pd_{v'}^m:n-m=h,
v\in X,v'\in X'\}
\end{equation}
have the same linear span in the Weyl algebra of $X'$, the ring of differential
operators with polynomial coefficients on $X'$. 
Putting $h=k,k+1,...$ we deduce that 
\begin{equation}\label{Hprop3}
v^{m+r+k}\pd_{v'}^m\fb f\in\t{C}_0(X')\mbox{ for all }m\geq 0,r\geq 0, v'\in X',v\in X
\end{equation}
We then claim that for all $v'\in X'$ and $m\geq 0$,
\\(i)   $x'\mapsto\pd_{v'}^m\fb f(x')$ is continuous at all $0\neq x'\in X'$, and
\\(ii) $|\pd_{v'}^m\fb f(x')|\leq C(v',m)\|x'\|^{-k-m}$ for all $0\neq x'\in X'$.
\\(iii) $|\pd_{v'}^m\fb f(x')|$ is $\O(\|x'\|^{-h})$ as $x'\to\infty$ for
every $h\in\Z$.

Both (i) and (ii) are deduced from (\ref{Hprop3}) by putting
$r=0$, letting the $v$'s run through a basis of $X$, for any fixed
choice of $v'\in X'$ and $m\geq 0$. Whereas (iii) is deduced in the same manner
by letting $r$ go to infinity in (\ref{Hprop3}). Let $\phi\in\test(X')$
so that $\phi\inverse(1)$ contains a neighbourhood of zero in $X'$. Then
(i) and (ii) imply that $\phi\fb f $ belongs to $\HH_0^+(X')$
whereas (i) and (iii) imply that $(1-\phi)\fb f$ belongs to $\sxp$.

We have now shown $\implies$ of the proposition. 

For the reverse implication, 
it suffices to prove that $\s{F}_{B'} u\in\HH_{\infty}^+(X)$ whenever 
$u\in\HH_0^+(X')$. So let $u$ be such a distribution on $X'$. Let $p(u)$
be the supremum of the set of $p\in\R$ for which the inequality of
(\ref{Hdfn1}) is valid for $\|x'\|\to 0$. 

We shall deal first with the case: $p(u)>0$. Recall that the restriction of $u$ to $X'\setminus\{0\}$
is given by a \df function $f:X\setminus\{0\}\to\C$. The assumption
$p(u)>0$ implies that $f$ extends as a continuous function $f:X'\to\C$
such that $f(0)=0$. In particular, $f$ gives rise to a distribution on $X$,
which we once again denote by $f$. 
The distribution $w=u-f$ is supported at $0$. By a theorem of L. Schwartz
 (see [R2],thm.6.25, page 150, and the identity (\ref{weyl}) for tempered
distributions), 
the Fourier transform of $w$ is a polynomial on $X$ and therefore belongs to
$\HH_{\infty}^+(X)$. So it remains to show that $\s{F}_{B'}f$ is in
$\HH_{\infty}^+(X)$. Because $f$ is compactly supported, its Fourier
transform is \df. So it remains only to verify the bounds on  $\s{F}_{B'}f$ imposed by (\ref{Hdfn1}).
 The bounds (\ref{Hdfn2}) placed on $u$ imply that $v^mf$ is $m$-times continuously differentiable on $X'$,
for all $v\in X$.  It follows that the distribution $\pd_{v'}^mv^mf$ is a continuous function
for all $m\geq 0,v\in X,v'\in X$, and therefore in $\abss(X')$, being
compactly supported. The identities (\ref{weyl}) now show that 
$\pd_v^m\s{F}_{B'}(x)$ is $\O(\|x\|^{-m})$ as $\|x\|\to\infty$, as required.

The general case $p(u)+k>0$ is dealt with by induction on $k\geq 0$.
From Weyl's commutation relations, one sees that 
$p(v.u)\geq 1+p(u)$ for all $v\in X$. Thus we may assume that $\s{F}_{B'}(v.u)$ is in $\HH_{\infty}^+(X)$ for all $v\in X$. In other words, $\pd_v\s{F}_{B'}u$
belongs to $\HH_{\infty}^+(X)$ for every $v\in X$. From this, it
is immediate that $\s{F}_{B'}u$ is itself in $\HH_{\infty}^+(X)$. This
completes the proof.
\end{proof}
\begin{rem}\label{obs} 
We intend to express some distributions in the form
 encountered in Proposition \ref{triple}. For this, it is
useful to note that 
\[v+f+\s{F}_{B'}v'=v_1+f_1+\s{F}_{B'}v'_1\implies v-v_1\in\test(X)\mbox{ and }
v'-v'_1\in\test(X')\] 
where it is assumed that $v,v_1\in\s{D}_c(X),v',v'_1\in\s{D}_c(X'),f,f_1\in\sx$.  
Indeed, by Paley-Wiener-Schwartz, $\s{F}_{B'}(v'-v'_1)$ is \df, and so it
follows that $v-v_1$ is \df as well. The same argument  after an application
of $\fb$ to both sides shows that $v'-v'_1$ is \df as well.
\end{rem}

\begin{prop}\label{Htriple}
\noindent
\begin{enumerate}
\item $\HH_{\infty}(X)=\s{D}_c(X)+\sx+\s{F}_{B'}\HH_0^+(X')$.
\item  $\HH(X)=\HH_0^+(X)+\sx+\s{F}_{B'}\HH_0^+(X')$.
\item $\HH(X)=\HH_{\infty}(X)\cap\s{F}_{B'}\HH_{\infty}(X')$.
\item $\fb\HH(X)=\HH(X')$.
\end{enumerate}
\end{prop}
\begin{proof} We first note that
\begin{enumerate}
\item[(a)] $\HH_{\infty}^+(X)=\sx+\s{F}_{B'}\HH_0^+(X')$
\item[(b)] $\HH_{\infty}(X)=\HH_{\infty}^+(X)+\s{D}_c(X)$
\item[(c)] $\HH(X)=\HH_0^+(X)+\HH_{\infty}^+(X)$.
\end{enumerate} 
Both (b) and (c) are obtained by writing $u=\phi.u+(1-\phi).u$ for
suitable test functions $\phi\in\test(X)$, and (a) is just Proposition \ref{Hprop}. Now part (1) follows from (a) and (b), and part (2) follows from
(a) and (c). 

We come to part (3). By part (1), we may express $u\in\HH_{\infty}(X)$ as $u=v+f+\s{F}_{B'}v'$
where $v\in\s{D}_c(X),f\in\sx,v'\in\HH_0^+(X')$. Now, if $u$ also belongs
to $\s{F}_{B'}\HH_{\infty}(X')$, we see that $u=v_1+f_1+\s{F}_{B'}v'_1$
 with $v_1\in\HH_0^+(X),f_1\in\sx, v'_1\in\s{D}_c(X')$. By remark  \ref{obs}, we see that $v-v_1\in\test(X)$. Because $v_1$ belongs to $\HH_0^+(X)$, 
we see that $v$ belongs to the same space as well. It follows that
$\HH_{\infty}(X)\cap\s{F}_{B'}\HH_{\infty}(X')$ is the same as
$\HH_0^+(X)+\sx+\s{F}_{B'}\HH_0^+(X')$.  Part (3) now follows from part (2).
Part (4) is immediate from part (3).

\end{proof}
\subsection{$T_B$ and $T_B^{reg}$ for mild singularities}
\noindent
\\Every $\phi$ in the Schwartz space of $X$ gives rise to the \df function $T_B\phi$ on $X\times X'$ (see prop.\ref{FPSF}),
and every tempered distribution $u$ on $X$ gives rise to the distribution $T_Bu$ on $X\times X'$ ( see prop. \ref{FPSF2}). 
Every $u\in\HH_{\infty}(X)$ is a tempered distribution on $X$; therefore $T_Bu$ is defined, and we wish to study it.
The next lemma deals with the preliminary case where $u=f\in\HH_{\infty}^+(X)$.

\begin{lem}\label{mainthmpre}Let $f\in\HH_{\infty}^+(X)$. Then  
$T_Bf$ is continuous on $X\times(X'\setminus\G')$. Furthermore, $T_Bf(x,x')$
 is simply the sum of the t-summable series  $\sum{\g\in\G} f(x+\g)\psi_B(\g,x')$
for every $x\in X,x'\in X',x'\notin \G'$. 

\end{lem}
\begin{proof}
\sloppypar
The action of $X$ on $\sx$ and $\smdouble$ given by
translations: \[\t{L_v}f(x)=f(x-v)\mbox{ and }\t{L_v}h(x,x')=h(x-v,x')
\,\,\forall x,v\in X,\forall x'\in X'\] for all $f\in\sx,h\in\smdouble$
extends in the standard manner to an action on $\sx^*$ and $\disdouble$
respectively. Furthermore $T_B\circ L_v=L_v\circ T_B$ for all $v\in X$.
It follows that $T_B$ commutes with $(1-L_v)^m$ for all $m\geq 0$.

Now let $f\in\HH_{\infty}^+(X)$. Choose $m\geq 0$ so that 
$\pd_v^mf$ is $\O(\|x\|^{-(1+\dim X)})$.
Employing (\ref{ftcc}), such upper bounds are valid for 
$g=(1-L_v)^mf$ as well. 
 Lemma \ref{consistency} shows that the distribution $T_Bg$ is a
continuous function, and also that $T_Bg(x,x')$ is the sum of the 
absolutely convergent series $\Sigma_{\G}g(x+\g)\psi_B(\g,x')$ for
all $x\in X,x'\in X'$.

Now take $v\in\G$. Let $w(x,x')=\psi_B(v,x')$ for all $(x,x')\in X\times X'$.
By (\ref{tinvariance}), we see that $T_BL_vu=w.T_Bu$ for all $u\in\sx^*$. It
follows that $T_Bg=(1-w)^mT_Bf$ .The continuity of $T_Bf$ on the region $w\neq 1$ follows.
That the series in lemma \ref{mainthmpre} is t-summable has already been
remarked in thm. \ref{thmtcon}. The sum of this t-summable series
is $(1-\psi_B(v,x'))^{-m}T_Bg(x,x')$ when $\psi_B(v,x')\neq 1$, by its
definition. Now, if $x'\notin\G'$, there is some $v\in\G$ for which
$\psi_B(v,x')\neq 1$. This proves the lemma. 

That $T_Bf$ is \df can be proved by the same method. But it is also
a  consequence of the theorem below because $\HH_{\infty}^+(X)\subset
\HH(X)$.

\end{proof}

The definition of $T^{reg}_Bu$ in the next thm. has been  given in definition \ref{regdfn}.
\begin{thm}\label{mainthm} Let $u\in\HH(X)$. Let $f:X\setminus\{0\}\to\C$
(resp. $f':X'\to\C$) be the \df function obtained by restricting $u$ (resp.
$\fb u$) to the complement of $0$ in $X$ (resp in $X'$). 
Let $U=X\setminus(\G\setminus\{0\})$ and let  
 $U=X'\setminus(\G'\setminus\{0\})$.  Then  
\begin{enumerate}
\item $T_B^{reg}u$ is defined
\item $T_B^{reg}|_{U\times U'}$ is \df.
\item Furthermore $T_B^{reg}u(x,x')$ equals the sum of both the t-summable series
below 
\\(a) $-\psi_B(-x,x')f'(x')+\sum{0\neq\g\in\G}f(x+\g)\psi_B(\g,x')$
 for all $x\in U,x'\notin\G'$
\\(b) $ -f(x)+\sum{0\neq\g'\in\G'}f'(x'+\g')\psi_{B}(-x,x'+\g')$ for all $x\notin\G,
x'\in U'$.

\end{enumerate}

\end{thm}
\begin{proof} By \ref{Htriple}(2), we see that $u$ has the form encountered
in \ref{triple}. The singular supports of $u$ and $\fb u$ are contained
in $\{0\}$. So parts (1) and (2) follow from \ref{triple}. 

The above lemma shows that $T_B^{reg}u$ is given by the sum in (a) 
when $u\in\HH_{\infty}^+(X)$. If $u\in\HH_0^+(X)$, the sum in (a) is finite
and thus equals $T_B^{reg}$ by thm. \ref{breakup}. It is clear that
$\HH_{\infty}^+(X)+\HH_0^+(X)=\HH(X)$. Therefore it has been proved that the sum 
in (a) equals $T_B^{reg}u$ for all $u\in\HH(X)$. The remaining equality
is obtained by interchanging the roles of $u$ and $\fb u$.

\end{proof}
\begin{rem} With $U$ as given in prop. \ref{triple} and $U'$ as above,
it is clear that the equality of $T_B^{reg}$ with the sum in (a) holds
also when $u\in\HH_{\infty}(X)$. 
\end{rem}
\begin{dfn}\label{value}
 We define $\s{L}u=T_B^{reg}u(0,0)$ for $u\in\HH(X)$
and   $\s{L}u'=T_{B'}^{reg}u'(0,0)$ for $u\in\HH(X)$. We note
that $\s{L}u=\s{L}\fb u$ and also that $\s{L}$ is invariant
under the action of $\t{Aut}(\G)$. 
\end{dfn}
\subsection{Proof of Theorem \ref{mainthmintro}}
\noindent
\\
Let us return to the situation of Theorem \ref{mainthmintro}.  Thus we are given a \df function $f:X\setminus\{0\}\to\C$ ,
homogeneous of type $(-s,\epsilon)\in\C\times\{\pm 1\}$, i.e. 
\[f(tx)=t^{-s}f(x)\text{ and }f(-x)=\epsilon f(x)\text{ for all }0\neq x\in X\text{ and }t>0\]
where it is assumed that 
\\(a) $(s,\epsilon)\notin\{(m+\dim(X),(-1)^m):m\geq 0,\,m\in\Z\}$
\\(b) $(s,\epsilon)\notin\{(-m,(-1)^m): m\geq 0,\,m\in\Z\}$

Condition (a) guarantees that $f$ extends uniquely to a distribution $u$ on $X$ with the same homogeneity degree. We see that
$u\in\HH(X)$. Its Fourier transform $\fb u$ belongs to $\HH(X')$ and is homogeneous of type $(s-\dim(X),\epsilon)$. 
The restriction of $\fb u$ to $X'\setminus\{0\}$ is a \df function, which will be denoted by $\fb f$ for simplicity. Condition (b) 
ensures that $\fb u$ is the unique homogeneous distribution on $X'$ that extends the \df function $\fb f$. 
Taking the \df function $F_f^{reg}$ of thm. \ref{mainthmintro} to be $T_B^{reg}u$, we see that   part (3) of thm. \ref{mainthm}
implies part (1) of thm.\ref{mainthmintro}.
The next lemma proves the rest  of thm. \ref{mainthmintro}.

\begin{lem} The sum of the h-summable series $\Sigma\{f(\g): 0\neq \g\in\G\}$ equals $\s{L}(u)$ as given in defn.\ref{value}.

\end{lem}
\begin{proof}

The action
of $g\in\t{GL}(X)$ on distributions on $X,X',X\times X'$ is denoted by
$\rho(g),\rho'(g),\rho''(g)$ respectively, and 
$\lambda\mapsto\lambda_X$ denotes the inclusion $\R\units\hkr\t{GL}(X)$.

Let $N\in\N$. For 
a distribution $w$ on $X\times X'$ or $X'$ that is invariant under
translation by $\G'$, let $U_Nw$ be the sum of its translates over
all the $N$-torsion points of $X'/\G'$. For $u\in\sx^*$ and $u_{\g}$
 as in \ref{breakup} we note that $U_Nu_{\g}=N^{\dim X}u_{\g}$ if
$\g\in N\G$ and zero otherwise. Summing over $\G$ and simplifying,
we see:
\\(a)  $U_NT_Bu=N^{\dim X}\rho''(N_X)T_B\rho(N_X)\inverse u$.

 The terms
$u_{\g}$ for $\g=0$ (resp $u'_{\g'}$ for $\g'=0$ when $u'\in\sxp^*$) 
will be denoted by $u(x,0)$ (resp. $u'(0,x')$). Let $T_B'u=T_Bu-u(x,0)$.
We see that (a) can be rewritten as:
\\(b) $U_NT_B'u=N^{\dim X}\rho''(N_X)T_B'\rho(N_X)\inverse u$.

Now write $U_Nw=w+U_N'w$. Subtract $\psi_B(-x,x')\fb u (0,x')$
from both sides of (b) when $u\in\HH(X)$. We obtain:
\begin{equation}\label{atlast}
U_N'T_B'u+T_B^{reg}u=N^{\dim X}\rho''(N_X)T_B^{reg}\rho(N_X)\inverse u
\mbox{ for all }u\in\HH(X).
\end{equation}
All the three terms above are \df at $(0,0)$ and evaluation at this point
gives
\begin{equation}\label{final}
U_N'T_B'u(0,0)+\s{L}u=N^{\dim X}\s{L}\rho(N_X)\inverse u\mbox{ for all }u
\in\HH(X).
\end{equation}
Part 3(a) of Thm. \ref{mainthm}  has a sum indexed by $0\neq\g\in\G$, and this
is precisely $T'_Bu(x,x')$. Because $U_N'T_B'u(0,0)$ is the sum of 
the $T'_Bu(0,x')$ taken
over all the nontrivial $N$-torsion points $x'$ of $X'/\G'$ we see that (\ref{final})
is equivalent to (\ref{hintro})  under the homogeneity assumption $\rho(\lambda_Xu)=\lambda^{s}u$
 for all positive real numbers $\lambda$.
\end{proof}
\subsection{translation-invariant integrals and sums}
\noindent
\\ We define $\int_Xu$ for certain distributions $u$ on $X$. Proposition \ref{strong} shows that this definition contains 
t-summability.
\begin{dfn}\label{tplus}
 A distribution $u$ on $X$ \textbf{is t-integrable} if there is a 
pair $(\phi,f)$ with $\phi\in\test(X),f\in\sx$ that satisfy
\begin{enumerate}
\item $\int_X\phi\neq 0$ and $\int_Xf\neq 0$, 
\item $\phi*u\in\sx^*$ and $f*(\phi*u)\in\abss(X)$.
\end{enumerate}
$\int_Xu$ is then defined as $(\int_X\phi.\int_Xf)\inverse\int_Xf*(\phi*u)$
and is seen to be independent of the choice of $(\phi,f)$ thanks to
commutativity.
\end{dfn}
\begin{dfn}
A function $a:\G\to\C$ is \textbf{$\t{t}^+$summable} if there is a pair
$(c,g)$ of $\C$-valued functions on $\G$, where $c$ has finite support,
and $g$ is rapidly decreasing i.e. $|g(\g)|$ is $\O(\|\g\|^{-r})$ for all $r$,
with 
\\(i) $\Sigma_{\G}c\neq 0$ and $\Sigma_{\G}g\neq 0$
\\(ii) $c*a:\G\to\C$ has polynomial growth, and
 $g*(c*a)$ belongs to $\abss(\G)$.
\\
$\Sigma_{\G}a$ is then defined to be $(\Sigma_{\G}c)^{-1}(\Sigma_{\G}g)^{-1}\Sigma_{\G}g*(c*a)$.
\end{dfn}
\begin{prop}\label{strong}
If $a:\g\to\C$ is t-summable, then it is  $\t{t}^+$summable.
\\ If $a$ is $\t{t}^+$summable, then
the distribution $E(a)=\sum{\g\in\G}a(\g)\delta_{\g}$
(where $\delta_{\g}$ is the Dirac distribution at $\g$) on $X$ is t-integrable.
\end{prop} 
\begin{proof}
The first assertion is clear. 

For the second assertion,
choose any $h\in\test(X)$ with $\int_Xh\neq 0$. If the
pair $(c,g)$ proves the $\t{t}^+$summability of $a$,then 
the pair $(\phi,f)$ given by $\phi=h*E(c)$ and $f=h*E(g)$ proves
the t-integrability of $E(a)$. 
\end{proof}


\section{Groups acting on sets}
 \subsection{Noncommutative Canonical Extension} 
 \noindent
 \\
We improve the canonical extension of section 2.1 to include a possibly noncommutative setting. Our data now is:
\\(a) an associative ring $A$ (that is not assumed to be  commutative), \linebreak
a \emph{surjective} ring homomorphism $\epsilon:A\to k$ where $k$ is a (commutative) field, with $\textbf{m}= \ker(\epsilon)$,
\\(b) an inclusion $M\hookrightarrow N$ of left $A$-modules, and 
\\(c) a left $A$-module homomorphism $I:M\to k$.

The \emph{canonical extension} associated to this data is the left $A$-module homomorphism $I_c:M_c\to k$, where $M_c$ is a $A$-submodule
of $N$. In order to construct the canonical extension, we require the collections $\s{W}$ and $\s{W}_1$.
\\$\s{W}=\{W:M\subset W\subset N, W\mbox{ is a left }$A$-\mbox{submodule, and }
\t{Hom}_A(W'/M,k)=0\mbox{ for all }$A$-\mbox{left submodules } M\subset W'\subset W\}$. 
\\$\s{W}_1$ the collection of $W\in\s{W}$ for which there
is an left $A$-module homomorphism $I_W:W\to k$ satisfying $I_W|_M=I$. Evidently, if such an $I_W$ exists, it is unique, because 
$\Hom_A(W/M,k)=0$.

We first check that if $W_1,W_2\in\s{W}$, then $W_1+W_2$ also belongs to
$\s{W}$. Indeed, if $M\subset W'\subset W_1+W_2$, let $D=W'\cap W_1$. Note
that $D/M$ and $W'/D$ are  subquotients of $W_1/M$ and $W_2/M$ respectively.
It follows that $\t{Hom}_A(D/M,k)=\t{Hom}_A(W'/D),k)=0$ and this shows
that $\t{Hom}_A(W'/M,k)=0$ as well. Thus $W_1+W_2\in\s{W}$. By induction,
it follows that $\s{W}$ is closed under finite sums. Zorn's lemma implies 
that $\s{W}$ is closed under arbitrary sums.

We continue by showing that $\s{W}_1$ has the same property. Let $W_1,W_2\in
\s{W}_1$. By definition, we have $A$-module homomorphisms 
$I_1:W_1\to k$ and $I_2:W_2\to k$ that satisfy $I_1|_M=I_2|_M=I$. Let $W'=W_1\cap W_2$. We see that $W'\in\s{W}_1$.  The uniqueness
of $I_{W'}$ shows that $I_1|_{W'}=I_2|_{W'}$. We deduce that
there is a $\tilde{I}:W_1+W_2\to k$ that extends both $I_1$ and $I_2$. 
Therefore $W_1+W_2$ also belongs to $\s{W}_1$. We deduce that $\s{W}_1$ is closed under arbitrary sums.
\begin{dfn}\label{noncomdfn}
$M_c:=\Sigma\{W:W\in\s{W}_1\}$ and $I_c:M_c\to k$ is the unique $A$-module homomorphism that restricts to the given $I:M\to k$.
\end{dfn}

The lemma below is useful for the proposition that follows. But the lemma follows straight
from the definitions, and so we skip its proof.
\begin{lem}\label{goose} 
Let $W\in\s{W}_1$ and let $I_W:W\to k$ be the unique left $A$-module
homomorphism that extends $I:M\to k$. Then the canonical extension $(W_c, (I_W)_c)$
associated to the data $(A,\epsilon:A\to k, W\hkr N, I_W:W\to k)$
 is the same as $(M_c,I_c)$.
\end{lem}

 We have seen that $n\in N$ belongs to $M_c$ if
and only if $M+An\in\s{W}_1$. We spell out this condition as explicitly as we
can below. 

For $n\in N$, let $J_n\subset A$ be the left ideal that annihilates
its image $\o{n}\in N/M$. Let $i_n:A/J_n\to N/M$ be the homomorphism that
sends $1$ to $\o{n}$.
The short exact sequence $0\to M\to N\to N/M\to 0$ gives 
$\xi\in\t{Ext}^1_A(N/M,M)$. The functoriality of $\t{Ext}$ in both variables,
given $I:M\to k$ and the above $i_n:A/J_n\to N/M$ produces an
element $\theta_n\in\t{Ext}^1_A(A/J_n,k)$.  

We deduce that 
$n\in\s{W}$ if and only if $\t{Hom}_A(A/J_{an},k)=0$ for all $a\in A$.
And $n\in\s{W}_1$ if and only if $\theta_n$ vanishes, in addition.

 By the long exact
sequence of $\t{Ext}_A(\cdot,k)$ one identifies 
$\t{Hom}_A(A/J_n,k)$  and $\t{Ext}^1_A(A/J_n,k)$ with 
$\t{Hom}_k(A/\textbf{m}+J_n,k)$
and $\t{Hom}_k((\textbf{m}\cap J_n/\textbf{m}J_n),k)$ respectively.
Under this identification, the composite $\textbf{m}\cap J_n\to
\textbf{m}\cap J_n/\textbf{m}J_n\xr{\theta_n}k$ is easily seen to
be $a\mapsto I(an)$ for all $a\in \textbf{m}\cap J_n$.

Let $S$ be the complement of $\textbf{m}$ in $A$. 
The condition $\t{Hom}_A(A/J_n,k)=0\iff J_n+\textbf{m}=A\iff S\cap J_n\neq\emptyset$. The statement below summarises this discussion.

\begin{lem}\label{need}
Let $n\in N$. Then $n\in M_c$ if and only if 
\\ \emph{(a)} $J_{an}+\textbf{m}=A$ for all $a\in A$, and
\\ \emph{(b)} $I(an)=0$ for all $a\in\textbf{m}\cap J_n$.
\end{lem}
When $A$ is commutative, the condition $J_n+\textbf{m}=A$
suffices for (a) because $J_{an}$ contains $J_n$ for all $a\in A$.
It also suffices for (b) because $\textbf{m}\cap J_n=\textbf{m} J_n$.
We deduce that the $(M_c,I_c)$ given
here in Definition \ref{noncomdfn} is consistent with that of  Section  \ref{canext}. 

We continue with $A=\C[G]$ and $\epsilon:A\to\C$ the augmentation homomorphism
as before. The canonical extension $(M_c,I_c)$ will now be denoted by
$(M_G,I_G)$. If $H$ is a subgroup, we also get the canonical extension
when $A$ is replaced by $\C[H]$, and this will naturally be denoted by
 $(M_H,I_H)$.
\begin{prop} Assume that $H$ is a normal subgroup of $G$. Then
\begin{enumerate}
\item $M_H$ is a $G$-submodule of $M_G$ and $I_G|_{M_H}=I_H$.
\item $(M_G,I_G)$ as defined above is also the canonical extension associated
to the data $(\C[G],M_H\hkr N,I_H:M_H\to\C)$.

\end{enumerate}
\end{prop}
\begin{proof} It is evident that $M_H$ is a $G$-module and that $I_H$ 
is a $G$-invariant linear functional. By assumption, there are no 
nonzero $H$-invariant linear functionals on any $H$-module $W$,
with $M\subset W\subset M_H$. It follows that
there are no 
nonzero $G$-invariant linear functionals on any $G$-module $W$,
with $M\subset W\subset M_H$. It follows that $M_H\in\s{W}_1$. The rest
of the proposition follows from \ref{goose}.

\end{proof}
From now on, we concentrate on the case: $X$ is a set equipped with $G$-action,
$M=\abss(X)\subset N=\C^X$ and $I:M\to\C$ is $\Sigma_X$. When $f:X\to\C$
belongs to $M_G$, we will say the series $\Sigma_Xf$ is $G$-summable.

Very specifically, we take $X=\Q^k$ (endowed with the discrete topology) equipped with the action of 
$G=\Q^k\rtimes\t{GL}_k(\Q)$.  
The center of
$\t{GL}_k(\Q)$ is denoted simply by $\Q\units\subset\t{GL}_k(\Q)$. 
We put $H=\Q^k\rtimes\Q\units\subset G$.  
In view of the remark following lemma \ref{need}, the preceeding proposition, and the normality of $\Q^k$
 and $H$ in $G$, we have the chain
of spaces
\[\abss(\Q^k)\subset\ts(\Q^k)=\abss(\Q^k)_{\Q^k}\subset\abss(\Q^k)_H\subset\abss(\Q^k)_G.
\]
where $\ts(\Q^k)$ is the space of t-summable functions.

The precise relation between $H$-summability as above and h-summability in the section 4 has not been worked out. The  proposition below
is a lengthy verification, relying on the claim that follows--the details are omitted.
\begin{prop}\label{Handh}
Let $f:\R^k\setminus\{0\}\to\C$ be a \df function homogeneous
of type $(-s,\epsilon)$. Assume $f\neq 0$. Then the series $\Sigma\{f(x):0\neq x\in\Z^k\}$  is $H$-summable if
and only if
\\(i) $(-s,\epsilon)\notin\{ (-k+m,(-1)^m):m\geq 0,m\in\Z\}$, or (ii) $f$ is a polynomial.

\end{prop}
\emph{Claim}. Restrict the $f$ given in the proposition to $\Z^k-\{0\}$ and then extend it by zero to get $g:\Q^k\to\C$. Then $\C[H]g+\abss(\Q^k)$ 
equals $B+\abss(\Q^k)$, where $B$ is the linear span of $\{P(D)f|_{x+N\Z^k}\}$ 
taken over all constant coefficient differential operators $P(D)$, and  all $x\in\Q^k$ and all natural numbers $N$.

\begin{ex}
We take $m=1$. For $k=1$, this says that $h:\Q\to\C$ is not H-summable, where given by 
$h(x)=x/|x|$ if $0\neq x\in\Z$, and $h(x)=0$ for all other rational numbers $x$
is not $H$-summable. This may also be seen very directly in the following
manner. Assume the contrary. Then we see that 
\[\Sigma_{\Q}h=\sum{x\in\Q}h(x)=\sum{x\in\Q}h(ax+b)
\]
for all $a\in\Q\units, b\in\Q$. Taking $a=-1,b=0$ we see that $\Sigma_{\Q}h=0$.
Taking $a=1,b=1$ we see that the sum of the series $x\mapsto h(x+1)-h(x)$
, indexed by $x\in\Q$ is zero, which contradicts the fact that this sum is 2!

For $m=1,k=2$, the function $f(x,y)=x/(x^2+y^2)$ has already occurred in
the counterexample given in the proof of \ref{special}(2). 
 \end{ex}

\noindent
\textbf{Acknowledgments}. The author thanks David Mumford for introducing him 
to the space $\smdouble$ in a beautiful course of lectures in 1979.
He also thanks C.Kenig, J.Lagarias, W.Li, S.Miller, R.Narasimhan, 
W.Schlag, W.Schmid and D. Thakur for useful comments
and assistance with references. 
\\The author also expresses his profound gratitude to the reviewer for going through the paper carefully and  making several recommendations to improve
its readability.


\begin{thebibliography}{99}
\bibitem{A}
Apostol, T.M.``Mathematical analysis''
Addison-Wesley series in mathematics, 1974.

\bibitem{A1}
Apostol, T.M. ``On the Lerch Zeta function''
Pacific J. Math. Vol.1 No.2 (1951), 161-167.

\bibitem{B}
Burnol, J-F
``Entrelacement de co-Poisson''
Ann. Inst. Fourier (Grenoble) 57 (2007),no.2, 525-602 

\bibitem{BP}
Burnol, J-F
``On Fourier and zeta(s)''
Forum Math. 16 (2004), no. 6, 789?840.

\bibitem{CF}Cassels,J. and Frohlich,A.``Algebraic Number Theory'',
Academic Press, 1967.
\bibitem{C}Connes A.``Trace formula in non-commutative Geometry and the zeros of the Riemann zeta function''
 Selecta Math. (N.S.) 5 (1999), 29?106

\bibitem{GL}Gerardin,P. and Li,W.``Twisted Dirichlet Series and Distributions'' Composition Math., tome 73, no.3(1990), p.271-293.



\bibitem{H} Hormander, L.
 	``The analysis of linear partial differential operators''
Grundlehren 256,Springer-Verlag, 1990.

\bibitem{Ha} Hardy,G.H.``Divergent Series'' Oxford, Clarendon Press, 1949.

\bibitem{KR} Kudla, S. and Rallis, S.``A regularized Siegel-Weil formula: The first term identity'' Annals of Math., 140(1994), p.1-80

\bibitem{L}
Lang, S. ``Algebraic Number Theory'' GTM 110, Springer-Verlag 1994.



\bibitem{LL}Lagarias,J. and Li,W.``The Lerch Zeta function I. Zeta Integrals.'' Forum Mathematicum, vol.24, 2012, 1-48.


\bibitem{LV}
Lion, G. and Vergne, M.
``The Weil representation, Maslov index, and theta series''
Progress in mathematics vol. 6, Birkhauser, 1980.



\bibitem{M}Mumford, D.``Tata lectures on theta'' I,III.
Progress in mathematics ; Birkhauser, 1983,1991.

\bibitem{MS} Miller, S, and Schmid, W. ``Distributions and analytic continuation of Dirichlet series''. J. Funct. Anal. 214 (2004), no. 1, 155-220.


\bibitem{S} Sondow, J.``
Analytic continuation of Riemann's zeta function and values at negative integers via Euler's transformation of series.''
Proc. Amer. Math. Soc. 120 (1994), no. 2, 421-424.



\bibitem{R}
Rudin, W.``Principles of mathematical analysis'' McGraw-Hill,1976.
\bibitem{R1} Rudin, W. ``Functional Analysis'' Tata McGraw Hill 1974.


\bibitem{T}
Titchmarsh, E.C.``The Theory of the Riemann Zeta Function'' 
Oxford Science Publications 1986.






\bibitem{BNT}

Weil, A. ``Basic Number Theory'' Grundlehren Bd 144,Springer-Verlag 1974.

\bibitem{acta}

Weil, A. 
``Sur certains groupes d'opérateurs unitaires.''
Acta Math. 111(1964)p143 -211. 

	




\end{thebibliography}
\end{document}